\newif\iflibertine
\libertinetrue
\pdfoutput=1
\documentclass{article}
\usepackage{import}
\usepackage{amsmath, amsthm, amssymb}
\usepackage{leftindex}
\usepackage{eucal}
\usepackage{mathrsfs}
\usepackage{geometry}
\usepackage[bottom]{footmisc}
\usepackage[unicode,bookmarksnumbered=true,colorlinks=true,allcolors=blue,linktoc=all,hypertexnames=false]{hyperref}
\usepackage[capitalise,nameinlink]{cleveref}
\usepackage{tikz-cd}
\usepackage{tikz}
\usetikzlibrary{arrows.meta,calc,positioning}
\usepackage{stmaryrd}
\usepackage{listings}
\usepackage{etoolbox}
\usepackage{thm-restate}
\usepackage[nottoc,notlot,notlof]{tocbibind}
\usepackage{tocloft}
\usepackage{enumitem}
\usepackage{mathtools}
\usepackage{quiver}
\usepackage{caption}
\usepackage{titlesec}

\iflibertine
\usepackage[T1]{fontenc}
\usepackage{libertine}
\usepackage[libertine]{newtxmath}
\input{glyphtounicode}
\pdfglyphtounicode{u1D434}{0041} % A
\pdfglyphtounicode{u1D435}{0042} % B
\pdfglyphtounicode{u1D436}{0043} % C
\pdfglyphtounicode{u1D437}{0044} % D
\pdfglyphtounicode{u1D438}{0045} % E
\pdfglyphtounicode{u1D439}{0046} % F
\pdfglyphtounicode{u1D43A}{0047} % G
\pdfglyphtounicode{u1D43B}{0048} % H
\pdfglyphtounicode{u1D43C}{0049} % I
\pdfglyphtounicode{u1D43D}{004A} % J
\pdfglyphtounicode{u1D43E}{004B} % K
\pdfglyphtounicode{u1D43F}{004C} % L
\pdfglyphtounicode{u1D440}{004D} % M
\pdfglyphtounicode{u1D441}{004E} % N
\pdfglyphtounicode{u1D442}{004F} % O
\pdfglyphtounicode{u1D443}{0050} % P
\pdfglyphtounicode{u1D444}{0051} % Q
\pdfglyphtounicode{u1D445}{0052} % R
\pdfglyphtounicode{u1D446}{0053} % S
\pdfglyphtounicode{u1D447}{0054} % T
\pdfglyphtounicode{u1D448}{0055} % U
\pdfglyphtounicode{u1D449}{0056} % V
\pdfglyphtounicode{u1D44A}{0057} % W
\pdfglyphtounicode{u1D44B}{0058} % X
\pdfglyphtounicode{u1D44C}{0059} % Y
\pdfglyphtounicode{u1D44D}{005A} % Z

\pdfglyphtounicode{u1D44E}{0061} % a
\pdfglyphtounicode{u1D44F}{0062} % b
\pdfglyphtounicode{u1D450}{0063} % c
\pdfglyphtounicode{u1D451}{0064} % d
\pdfglyphtounicode{u1D452}{0065} % e
\pdfglyphtounicode{u1D453}{0066} % f
\pdfglyphtounicode{u1D454}{0067} % g

\pdfglyphtounicode{uni210E}{0068} % h
\pdfglyphtounicode{u1D455}{0068} % h

\pdfglyphtounicode{u1D456}{0069} % i
\pdfglyphtounicode{u1D457}{006A} % j
\pdfglyphtounicode{u1D458}{006B} % k
\pdfglyphtounicode{u1D459}{006C} % l
\pdfglyphtounicode{u1D45A}{006D} % m
\pdfglyphtounicode{u1D45B}{006E} % n
\pdfglyphtounicode{u1D45C}{006F} % o
\pdfglyphtounicode{u1D45D}{0070} % p
\pdfglyphtounicode{u1D45E}{0071} % q
\pdfglyphtounicode{u1D45F}{0072} % r
\pdfglyphtounicode{u1D460}{0073} % s
\pdfglyphtounicode{u1D461}{0074} % t
\pdfglyphtounicode{u1D462}{0075} % u
\pdfglyphtounicode{u1D463}{0076} % v
\pdfglyphtounicode{u1D464}{0077} % w
\pdfglyphtounicode{u1D465}{0078} % x
\pdfglyphtounicode{u1D466}{0079} % y
\pdfglyphtounicode{u1D467}{007A} % z

\pdfglyphtounicode{u1D6E2}{0391} % Alpha
\pdfglyphtounicode{u1D6E3}{0392} % Beta
\pdfglyphtounicode{u1D6E4}{0393} % Gamma
\pdfglyphtounicode{u1D6E5}{0394} % Delta
\pdfglyphtounicode{u1D6E6}{0395} % Epsilon
\pdfglyphtounicode{u1D6E7}{0396} % Zeta
\pdfglyphtounicode{u1D6E8}{0397} % Eta
\pdfglyphtounicode{u1D6E9}{0398} % Theta
\pdfglyphtounicode{u1D6EA}{0399} % Iota
\pdfglyphtounicode{u1D6EB}{039A} % Kappa
\pdfglyphtounicode{u1D6EC}{039B} % Lambda
\pdfglyphtounicode{u1D6ED}{039C} % Mu
\pdfglyphtounicode{u1D6EE}{039D} % Nu
\pdfglyphtounicode{u1D6EF}{039E} % Xi
\pdfglyphtounicode{u1D6F0}{039F} % Omicron
\pdfglyphtounicode{u1D6F1}{03A0} % Pi
\pdfglyphtounicode{u1D6F2}{03A1} % Rho
\pdfglyphtounicode{u1D6F3}{03F4} % capital theta symbol
\pdfglyphtounicode{u1D6F4}{03A3} % Sigma
\pdfglyphtounicode{u1D6F5}{03A4} % Tau
\pdfglyphtounicode{u1D6F6}{03A5} % Upsilon
\pdfglyphtounicode{u1D6F7}{03A6} % Phi
\pdfglyphtounicode{u1D6F8}{03A7} % Chi
\pdfglyphtounicode{u1D6F9}{03A8} % Psi
\pdfglyphtounicode{u1D6FA}{03A9} % Omega

\pdfglyphtounicode{u1D6FB}{2207} % nabla

\pdfglyphtounicode{u1D6FC}{03B1} % alpha
\pdfglyphtounicode{u1D6FD}{03B2} % beta
\pdfglyphtounicode{u1D6FE}{03B3} % gamma
\pdfglyphtounicode{u1D6FF}{03B4} % delta
\pdfglyphtounicode{u1D700}{03B5} % epsilon
\pdfglyphtounicode{u1D701}{03B6} % zeta
\pdfglyphtounicode{u1D702}{03B7} % eta
\pdfglyphtounicode{u1D703}{03B8} % theta
\pdfglyphtounicode{u1D704}{03B9} % iota
\pdfglyphtounicode{u1D705}{03BA} % kappa
\pdfglyphtounicode{u1D706}{03BB} % lambda
\pdfglyphtounicode{u1D707}{03BC} % mu
\pdfglyphtounicode{u1D708}{03BD} % nu
\pdfglyphtounicode{u1D709}{03BE} % xi
\pdfglyphtounicode{u1D70A}{03BF} % omicron
\pdfglyphtounicode{u1D70B}{03C0} % pi
\pdfglyphtounicode{u1D70C}{03C1} % rho
\pdfglyphtounicode{u1D70D}{03C2} % final sigma
\pdfglyphtounicode{u1D70E}{03C3} % sigma
\pdfglyphtounicode{u1D70F}{03C4} % tau
\pdfglyphtounicode{u1D710}{03C5} % upsilon
\pdfglyphtounicode{u1D711}{03C6} % phi
\pdfglyphtounicode{u1D712}{03C7} % chi
\pdfglyphtounicode{u1D713}{03C8} % psi
\pdfglyphtounicode{u1D714}{03C9} % omega

\pdfglyphtounicode{u1D715}{2202} % partial
\pdfglyphtounicode{u1D716}{03F5} % epsilon symbol
\pdfglyphtounicode{u1D717}{03D1} % theta symbol
\pdfglyphtounicode{u1D718}{03F0} % kappa symbol
\pdfglyphtounicode{u1D719}{03D5} % phi symbol
\pdfglyphtounicode{u1D71A}{03F1} % rho symbol
\pdfglyphtounicode{u1D71B}{03D6} % pi symbol
\fi

\crefformat{enumi}{(#2#1#3)}

\crefname{subsection}{Subsection}{Subsections}

\hypersetup{bookmarksdepth=2}
\newcommand{\subsectionnotoc}[1]{%
	\let\oldaddcontentsline\addcontentsline
	\renewcommand{\addcontentsline}[3]{}%
	\subsection{#1}%
	\let\addcontentsline\oldaddcontentsline
}

\newcommand{\subsecref}[1]{\hyperref[#1]{\S\ref*{#1}}}

\titleformat{\section}{\normalfont\LARGE\bfseries}{\thesection}{1em}{}
\titleformat{\subsection}{\normalfont\Large\bfseries}{\thesubsection}{0.7em}{}
\titleformat{\subsubsection}{\normalfont\large\bfseries}{\thesubsubsection}{0.5em}{}

\titlespacing{\paragraph}{\parindent}{0pt}{1em}

\newtheorem{theorem}{Theorem}

\newtheorem{thm}{Theorem}[subsection]
\newtheorem{lem}[thm]{Lemma}
\newtheorem{prop}[thm]{Proposition}
\newtheorem{cor}[thm]{Corollary}

\newtheorem{question}[thm]{Question}

\theoremstyle{definition}
\newtheorem{defn}[thm]{Definition}
\newtheorem{notn}[thm]{Notation}
\newtheorem{const}[thm]{Construction}

\newtheorem{remark}[thm]{Remark}
\newtheorem{example}[thm]{Example}
\newtheorem{warn}[thm]{Warning}

\newtheorem*{thm*}{Theorem}

\newtheoremstyle{namedassumstyle}
  {}{}                  % default space above/below
  {\normalfont}{}       % body font, indent
  {\bfseries}{.}        % head font, punctuation
  { }                   % normal interword space
  {\thmnote{#3}\thmnumber{ #2}}
\theoremstyle{namedassumstyle}
\newtheorem{namedassuminner}[thm]{}
\theoremstyle{definition}
\newenvironment{namedassum}[1]{\begin{namedassuminner}[#1]}{\end{namedassuminner}}
\newcommand{\assref}[1]{\hyperref[#1]{\nameref*{#1}~\ref*{#1}}}

\newcommand{\ncmd}{\newcommand}

\definecolor{DefColor}{rgb}{0.6,0.15,0.25}
\newcommand{\mdef}[1]{\textcolor{DefColor}{#1}}
\newcommand{\tdef}[1]{\mdef{\emph{#1}}}

\ncmd{\mathbfsf}[1]{\mathord{\text{\normalfont\bfseries #1}}}

\iflibertine
\ncmd{\mbb}[1]{\mathbfsf{#1}}
\else
\ncmd{\mbb}[1]{\mathbb{#1}}
\fi
\ncmd{\mrm}[1]{\mathrm{#1}}
\ncmd{\mcl}[1]{\mathcal{#1}}
\ncmd{\mfk}[1]{\mathfrak{#1}}
\ncmd{\mbf}[1]{\mathbf{#1}}
\ncmd{\mscr}[1]{\mathscr{#1}}

\ncmd{\todo}[1]{\textbf{TODO #1}}
\ncmd{\reftodo}[1]{\textbf{REF #1}}

\DeclareRobustCommand{\minwidthbox}[2]{%
	\mathmakebox[\ifdim#2<\width\width\else#2\fi]{#1}%
}
\ncmd{\too}[1][]{\xrightarrow{\minwidthbox{#1}{1em}}}
\ncmd{\oot}[1][]{\xleftarrow{\minwidthbox{#1}{1em}}}
\def\isoraise{\iflibertine -0.8ex\else -0.5ex\fi}
\ncmd{\iso}{\too[\smash{\raisebox{\isoraise}{\ensuremath{\scriptstyle\sim}}}]}
\ncmd{\osi}{\oot[\smash{\raisebox{\isoraise}{\ensuremath{\scriptstyle\sim}}}]}
\ncmd{\hoook}[1][]{\xhookrightarrow{\minwidthbox{#1}{1em}}}
\ncmd{\adj}[1][]{\mathrel{\substack{\xrightarrow{\minwidthbox{#1}{1em}} \\[-.7ex] \xleftarrow{\minwidthbox{#1}{1em}}}}}
\ncmd{\mapstoo}[1][]{\xmapsto{\minwidthbox{#1}{1em}}}

\ncmd{\qin}{\quad\in\quad}
\ncmd{\mmod}{/\mkern-4mu/}
\ncmd{\Id}{\mrm{Id}}
\ncmd{\Nm}{\mathrm{Nm}}
\ncmd{\cont}{\mrm{cont}}

\ncmd{\bbA}{\mbb{A}}
\ncmd{\bbB}{\mbb{B}}
\ncmd{\clA}{\mcl{A}}
\ncmd{\clB}{\mcl{B}}
\ncmd{\CC}{\mcl{C}}
\ncmd{\DD}{\mcl{D}}
\ncmd{\EE}{\mcl{E}}
\ncmd{\TT}{\mcl{T}}
\ncmd{\MM}{\mcl{M}}
\ncmd{\KK}{\mrm{K}}
\ncmd{\Kcont}{\KK^\cont}
\ncmd{\Uloc}{\mcl{U}_\mrm{loc}}
\ncmd{\Uloccont}{\Uloc^\cont}
\ncmd{\ku}{\mrm{ku}}
\ncmd{\KU}{\mrm{KU}}
\ncmd{\cX}{\mcl{X}}
\ncmd{\cY}{\mcl{Y}}
\ncmd{\kn}{\mrm{k}(n)}
\ncmd{\Kn}{\KK(n)}
\ncmd{\Knp}{\KK(n{+}1)}
\ncmd{\Ktn}{\KK_n}
\ncmd{\Ko}{\KK(1)}
\ncmd{\Tn}{\mrm{T}(n)}
\ncmd{\Tnp}{\mrm{T}(n{+}1)}
\ncmd{\To}{\mrm{T}(1)}
\ncmd{\Tm}{\mrm{T}(m)}
\ncmd{\Tmn}{\mrm{T}(m{-}1)}
\ncmd{\Tmp}{\mrm{T}(m{+}1)}
\ncmd{\Ti}{\mrm{T}(i)}
\ncmd{\KTnp}{\KK_{\Tnp}}
\ncmd{\KKo}{\KK_{\Ko}}
\ncmd{\KTo}{\KK_{\To}}
\ncmd{\THH}{\mrm{THH}}
\ncmd{\TC}{\mrm{TC}}
\ncmd{\TCm}{\TC^-}
\ncmd{\rmE}{\mrm{E}}
\ncmd{\En}{\rmE_n}
\ncmd{\Enp}{\rmE_{n{+}1}}
\ncmd{\JWn}{\rmE(n)}
\ncmd{\cJWn}{\widehat{\JWn}}
\ncmd{\BP}{\mrm{BP}}
\ncmd{\BPn}{\BP\langle n\rangle}
\ncmd{\Enpk}{\Enp(\kappa)}
\ncmd{\lsF}{\mscr{F}}
\ncmd{\lsG}{\mscr{G}}
\ncmd{\bbC}{\mbb{C}}
\ncmd{\bbE}{\mbb{E}}
\ncmd{\Eo}{\bbE_1}
\ncmd{\Et}{\bbE_2}
\ncmd{\Eth}{\bbE_3}
\ncmd{\GG}{\mbb{G}}
\ncmd{\NN}{\mbb{N}}
\ncmd{\Ng}{\NN_>}
\ncmd{\ZZ}{\mbb{Z}}
\ncmd{\QQ}{\mbb{Q}}
\ncmd{\Qbar}{\overline{\QQ}}
\ncmd{\clQ}{\mcl{Q}}
\ncmd{\Qab}{\QQ(\zeta_\infty)}
\ncmd{\Zp}{\ZZ_p}
\ncmd{\Qp}{\QQ_p}
\ncmd{\Fp}{\mbb{F}_p}
\ncmd{\Fpbar}{\overline{\mbb{F}}_p}
\ncmd{\Fell}{\mbb{F}_\ell}
\renewcommand{\SS}{\mbb{S}}
\ncmd{\SKn}{\SS_{\Kn}}
\ncmd{\SKnp}{\SS_{\Knp}}
\ncmd{\SKo}{\SS_{\Ko}}
\ncmd{\STn}{\SS_{\Tn}}
\ncmd{\STnp}{\SS_{\Tnp}}
\ncmd{\MU}{\mrm{MU}}
\ncmd{\rmP}{\mrm{P}}
\ncmd{\OO}{\mcl{O}}
\ncmd{\one}{\mbf{1}}
\ncmd{\ii}{\mbf{i}}
\ncmd{\jj}{\mbf{j}}
\ncmd{\kk}{\mbf{k}}
\ncmd{\xx}{\mbf{x}}
\ncmd{\yy}{\mbf{y}}
\ncmd{\ci}{\underline{i}}
\ncmd{\cbullet}{\underline{\bullet}}
\renewcommand{\AA}{\one}
\ncmd{\Ax}{\kappa}
\ncmd{\BB}{B}
\ncmd{\xc}[1]{\widehat{#1}}
\ncmd{\EEc}{\widehat{\EE}}
\ncmd{\CCc}{\widehat{\CC}}

\ncmd{\irchi}[2]{\raisebox{\depth/2}{$#1\chi$}}
\DeclareRobustCommand{\rchi}{{\mathpalette\irchi\relax}}
\ncmd{\ch}{\rchi}
\ncmd{\cch}{\widehat{\scalebox{1.15}{$\rchi$}}}

\ncmd{\otimesu}{\mathop{\otimes}\limits}
\ncmd{\otimesA}{\otimes}

\ncmd{\fL}{L}
\ncmd{\fLLam}{\fL^\Lambda}
\ncmd{\LKn}{L_{\Kn}}
\ncmd{\LKnp}{L_{\Knp}}
\ncmd{\LKo}{L_{\Ko}}
\ncmd{\LTn}{L_{\Tn}}
\ncmd{\LTnp}{L_{\Tnp}}
\ncmd{\LTm}{L_{\Tm}}
\ncmd{\Ln}{L_n}
\ncmd{\Lnm}{L_{n{-}1}}
\ncmd{\Mn}{M_n}
\ncmd{\Mnf}{M_n^f}
\ncmd{\Lnf}{L_n^f}
\ncmd{\Lnmf}{L_{n{-}1}^f}
\ncmd{\Mz}{M_0}
\ncmd{\Lz}{L_0}
\ncmd{\Gammac}{\Gamma_c}

\ncmd{\LL}{\mrm{L}}
\ncmd{\RR}{\mrm{R}}
\ncmd{\BC}{\mrm{BC}}
\ncmd{\psa}{\oplus}
\ncmd{\dbl}{\mrm{dbl}}
\ncmd{\dual}{\mrm{dual}}
\ncmd{\op}{\mrm{op}}
\ncmd{\pifin}{\pi\text{-}\mathrm{fin}}
\ncmd{\pfin}{p\text{-}\mathrm{fin}}
\ncmd{\seg}{\mrm{seg}}
\ncmd{\perf}{\mrm{perf}}
\ncmd{\aug}{\mrm{aug}}
\ncmd{\nonu}{\mrm{nu}}
\ncmd{\Eonu}{\Eo^\nonu}
\ncmd{\ldual}[1]{{}^\vee#1}

\ncmd{\pt}{\mrm{pt}}
\ncmd{\Perf}{\mrm{Perf}}
\ncmd{\Motloc}{\mrm{Mot}_\mrm{loc}}
\ncmd{\Mod}{\mrm{Mod}}
\ncmd{\RMod}{\mrm{RMod}}
\ncmd{\BMod}[2]{\leftindex_{#1}{\mrm{BMod}}_{#2}}
\ncmd{\cMod}{\widehat{\Mod}\vphantom{\Mod}}
\ncmd{\Vect}{\mrm{Vect}}
\ncmd{\Ab}{\mrm{Ab}}
\ncmd{\Fin}{\mrm{Fin}}
\ncmd{\Spaces}{\mcl{S}}
\ncmd{\Spacespifin}{\Spaces_{\pifin}}
\ncmd{\Spacespfin}{\Spaces_{\pfin}}
\ncmd{\Sp}{\mrm{Sp}}
\ncmd{\cSp}{\widehat{\Sp}\vphantom{\Sp}}
\ncmd{\Spcn}{\Sp_{\geq 0}}
\ncmd{\SpTn}{\Sp_{\Tn}}
\ncmd{\SpTnp}{\Sp_{\Tnp}}
\ncmd{\SpTm}{\Sp_{\Tm}}
\ncmd{\SpKn}{\Sp_{\Kn}}
\ncmd{\SpKo}{\Sp_{\Ko}}
\ncmd{\SpKnp}{\Sp_{\Knp}}
\ncmd{\Span}{\mrm{Span}}
\ncmd{\Spano}{\Span_1}
\ncmd{\Cat}{\mrm{Cat}}
\ncmd{\Catpifin}{\Cat_{\pifin}}
\ncmd{\Catpfin}{\Cat_{\pfin}}
\ncmd{\CatLn}{\Cat_{\Ln}}
\ncmd{\CatLnm}{\Cat_{\Lnm}}
\ncmd{\CatMn}{\Cat_{\Mn}}
\ncmd{\CatMnf}{\Cat_{\Mnf}}
\ncmd{\CatLnf}{\Cat_{\Lnf}}
\ncmd{\Catperf}{\Cat_{\perf}}
\ncmd{\CAT}{\widehat{\Cat}\vphantom{\Cat}}
\ncmd{\PrL}{\mrm{Pr}^\LL}
\ncmd{\PrLst}{\PrL_{\mrm{st}}}
\ncmd{\PrLstw}{\mrm{Pr}^{\omega}_{\mrm{st}}}
\ncmd{\PrLstdbl}{\mrm{Pr}^{\dbl}_{\mrm{st}}}
\ncmd{\PrLE}{\PrL_{\EE}}
\ncmd{\PrLEdbl}{\mrm{Pr}^{\dbl}_{\EE}}
\ncmd{\PrLR}{\PrL_R}
\ncmd{\Mor}{\mrm{Mor}}
\ncmd{\MorE}{\Mor(\EE)}
\ncmd{\Mfg}{\mcl{M}_\mrm{fg}}
\ncmd{\Thick}{\mrm{Thick}}
\ncmd{\Pro}{\mrm{Pro}}
\ncmd{\Nuc}{\mrm{Nuc}}
\ncmd{\NucE}{\Nuc(\EEc)}
\ncmd{\NucR}{\Nuc(\xc{R})}
\ncmd{\Gr}{\mrm{Gr}}
\ncmd{\GrN}{\Gr_{\geq 0}}
\ncmd{\GrZ}{\Gr}
\ncmd{\Gri}[1][i]{\Gr_{[0,#1)}}

\ncmd{\yon}{\text{\usefont{U}{min}{m}{n}\symbol{'110}}}
\DeclareFontFamily{U}{min}{}
\DeclareFontShape{U}{min}{m}{n}{<-> dmjhira}{}

\DeclareMathOperator{\fib}{fib}

\DeclareMathOperator{\free}{free}

\DeclareMathOperator{\Ind}{Ind}

\DeclareMathOperator{\End}{End}
\DeclareMathOperator{\Fun}{Fun}
\ncmd{\FunL}{\Fun^\LL}
\ncmd{\FunLE}{\FunL_\EE}
\DeclareMathOperator{\Alg}{Alg}
\ncmd{\Algaug}{\Alg^\aug}
\ncmd{\Algnu}{\Alg^\nonu}
\DeclareMathOperator{\CAlg}{CAlg}
\DeclareMathOperator{\CMon}{CMon}
\DeclareMathOperator{\coCMon}{coCMon}
\ncmd{\CMoninf}{\CMon_\infty}
\ncmd{\coCMoninf}{\coCMon_\infty}
\DeclareMathOperator*{\colim}{colim}

\ncmd{\limdual}[1][]{\lim_{#1}\nolimits^{\dbl}}
\ncmd{\limw}[1][]{\lim_{#1}\nolimits^{\omega}}
\ncmd{\prolim}[1][]{\operatorname*{``lim''}_{#1}}

\DeclareMathOperator{\Shv}{Sh}

\ncmd\noloc{%
	\nobreak
	\mspace{6mu plus 1mu}
	{:}
	\nonscript\mkern-\thinmuskip
	\mathpunct{}
	\mspace{2mu}
}

\title{\vspace{-1.2em}\makebox[\textwidth][c]{Algebraic K-theory of quotient ring spectra}}
\author{Shay Ben-Moshe\thanks{Max Planck Institute for Mathematics, Bonn, Germany.} \thanks{Faculty of Mathematics and Computer Science, Weizmann Institute of Science, Israel.}}
\date{}

\begin{document}
	\maketitle

	\vspace{-1.8em}
	
	\begin{abstract}
		We study the algebraic K-theory of quotients of even ring spectra. Building on Efimov's work, we prove that the limit of the K-theories of these quotients is the continuous K-theory of nuclear modules. For quotients by elements in the height $n$ chromatic ideal, we show that this agrees with the K-theory of the ring spectrum, after chromatic localization at heights $\geq n{+}1$.
	\end{abstract}

    \vfill

    % \begin{center}
    %     \includegraphics[width=140mm]{painting/painting.jpg}
		
	% 	\vspace{.6em}

	% 	\small
    %     \makebox[\linewidth][c]{\begin{tabular}{@{}l@{}}
	% 		Sol LeWitt, \emph{Incomplete Open Cubes 6/1, 9/4, 8/10, 8/19, 8/22, 10/1}, 1974.\\
	% 		Photo by K. Daem, Raas van Gaverestraat, Ghent, 1994.\\
	% 		Collection Herbert Foundation, Ghent, \iflibertine \textcopyright{} \else \raisebox{0.4ex}{\scalebox{0.6}{©}} \fi Herbert Foundation.
	% 	\end{tabular}}
    % \end{center}
	
	\tableofcontents

    \vfill

	\clearpage
	
	\section{Introduction}

\subsectionnotoc{Background}

Given an ideal $I = (x_1,\dotsc,x_r)$ in a ring $R$, it is natural to study $R$, and invariants attached to it, by its quotients or completion.
For algebraic K-theory, there are several competing (small stable) categories\footnote{We use the term category to refer to an $\infty$-category and $2$-category to refer to an $(\infty,2)$-category.} to consider, including perfect modules over the $I$-adic completion $\Perf(\xc{R})$, or $I$-adically complete modules $\cMod_R$ with the finiteness condition being either dualizability or compactness.
Another candidate is the limit of K-theories of the quotients
\[
	\KK(R) \too \lim \KK(R/(x_1^{i_1},\dotsc,x_r^{i_r})).
\]
For ordinary rings, both the map and its target have been studied extensively, and we refer the reader to \cite[\S5.1]{henselian-pairs} and \cite{dahlhausen2024continuous,DYZ} for detailed accounts of the literature and recent results.

In this paper we study these for ring spectra.
Quotients in higher algebra, however, are subtle, often failing to have a multiplicative structure.
Nevertheless, a construction for \emph{even} $\Et$-ring spectra is available, by insights of Hahn--Wilson \cite{E2quot,HW} and Ausoni--Bay{\i}nd{\i}r--Moulinos and Rognes--Sagave--Schlichtkrull \cite{ABM,RSS}, building on \cite{rotation}.\footnote{A different approach to quotients due to Burklund \cite{BurklundMoore} is discussed in more detail in \cref{related-work}.}
This relies on two fortunate miracles:
\begin{enumerate}
	\item The free $\Eo$-algebra on a generator in an even degree $d \in 2\ZZ$
	\[
		\SS[x] := \free_{\Eo}(\Sigma^{d}\SS)
		\qin \Alg_{\Eo}(\Sp)
	\]
	has a canonical $\Et$-algebra structure, which moreover respects the grading by powers of $x$.
	\item For an even $\Et$-ring $R$ and elements $x_1,\dotsc,x_r$ of non-negative even degrees, there is\footnote{The argument is by obstruction theory, so the $\Et$-map is not obtained canonically or guaranteed to be unique.} an $\Et$-map
	\[
		\SS[x_1] \otimes \cdots \otimes \SS[x_r] \too R
		\qin \Alg_{\Et}(\Sp).
	\]
\end{enumerate}
The first part endows the universal quotient $\SS[x]/x^i$ with an $\Eo$-$\SS[x]$-algebra structure, by truncating the grading.
Using the second part, this can be base-changed to endow the quotients $R/(x_1^{i_1},\dotsc,x_r^{i_r})$ with an $\Eo$-$R$-algebra structure.
We review these constructions in \cref{subsec-gr-poly,subsec-e2-quot}.

Importantly, even ring spectra are ubiquitous in chromatic homotopy theory, including for example the complex cobordism spectrum $\MU$, the Brown--Peterson spectrum $\BP$, and various variants of Morava E-theory such as $\En$, $\JWn$, $\cJWn$ and $\BPn$.
In fact, since $R$ is even, it is complex orientable, demonstrating that the connection to chromatic homotopy theory is not merely by way of examples.

In this paper we address the following two questions, raised by the above discussion, for quotients of even $\Et$-ring spectra, and, in particular, under chromatic assumptions.

% Moreover, complex orientability gives rise to the height $n$ ideal $(p,v_1,\dotsc,v_{n-1}) \subset \pi_*(R)$.
% In this paper, we shall address the following two questions raised by the above discussion, and in particular under the assumption that the elements $x_1,\dotsc,x_r$ are in the height $n$ ideal.
% The complex orientation gives rise to the height $n$ ideal $(p,v_1,\dotsc,v_{n-1}) \subset \pi_*(R)$.
% The main object of study of this paper is the map
% \[
% 	\KK(R) \too \lim \KK(R/(x_1^{i_1},\dotsc,x_r^{i_r})),
% \]
% under the assumption that the elements $x_1,\dotsc,x_r$ are in (the radical of) the height $n$ ideal $(p,v_1,\dotsc,v_{n-1}) \subset \pi_*(R)$.

% We note that these results do not make any completeness assumptions on $R$ or the modules over it.
% It may be more natural instead to study the algebraic K-theory of other (small stable) categories, such as perfect modules over the $I$-adic completion $\Perf(\xc{R})$, or $I$-adically complete modules $\cMod_R$ with the finiteness condition being either dualizability or compactness.
% In fact, this conceals two a priori separate questions.

\begin{question}\label{k-of-what}
	Is the limit of K-theories of quotients the K-theory of some category attached to $I \subset R$?
\end{question}

\begin{question}\label{is-iso}
	What is the relation between (the K-theory of) the different candidate module categories?
\end{question}

\subsectionnotoc{Main results}

\cref{k-of-what} was recently answered by Efimov for ordinary commutative rings.
As alluded to above, algebraic K-theory is typically applied to small stable categories, which are (up to idempotent completion) equivalent to compactly generated presentable stable categories.
In his influential work \cite{EfimovK}, Efimov has vastly expanded the collection of categories to which K-theory can be applied to all dualizable presentable stable categories, where it is denoted $\Kcont$.
These include various new examples of interest, such as categories of sheaves and local systems among others \cite[\S1.5]{EfimovK}.
In a sequel paper \cite{EfimovLimit}, Efimov shows that the limit of the K-theories of the quotients is the continuous K-theory of the category of nuclear modules,\footnote{A different version of nuclear modules was originally introduced by Clausen--Scholze. These make no appearance in the present paper, and we refer the reader to \cite{EfimovLimit} for more details.} defined to be the limit in dualizable categories
\[
	\NucR := \limdual \Mod_{R/(x_1^{i_1},\dotsc,x_r^{i_r})}
	\qin \PrLstdbl.
\]
We extend this result to quotients of even $\Et$-ring spectra, taking the same definition of nuclear modules.

\begin{theorem}[{\cref{k-nuc}}]\label{k-nuc-intro}
	For an even $\Et$-ring spectrum $R$ and elements $x_1,\dotsc,x_r$ of non-negative even degrees, the assembly map induces an isomorphism
	\[
		\Kcont(\NucR) \iso \lim \KK(R/(x_1^{i_1},\dotsc,x_r^{i_r})).
	\]
	More generally, for any dualizable $\DD \in \PrLstdbl$, the assembly map induces an isomorphism
	\[
		\Kcont(\NucR \otimes \DD) \iso \lim \Kcont(\Mod_{R/(x_1^{i_1},\dotsc,x_r^{i_r})} \otimes \DD).
	\]
\end{theorem}

\begin{remark}
	The quotients depend on a choice of an $\Et$-map $\SS[x_1] \otimes \cdots \otimes \SS[x_r] \to R$, but all the results hold for any such choice, which we therefore leave implicit in the introduction.
\end{remark}

Having introduced yet another candidate category of modules, we turn our attention to \cref{is-iso}.
Nuclear modules were defined as the limit of modules over the quotients, taken among dualizable categories.
Dualizable categories are, on the one hand, more general than compactly generated categories, and on the other hand, less general than presentable categories
\[
	\PrLstw \quad\subset\quad \PrLstdbl \quad\subset\quad \PrLst.
\]
We show that the ordinary and compactly generated limits give the (Ind-completion of) compact or dualizable complete modules.
Moreover, we compare the three categories, showing that they embed fully faithfully into one another, and control their Verdier quotients.
As we explain in \cref{outline}, while not needed for the statements, our proofs require $R$ to have an $\Eth$-ring spectrum structure.
We also refer the reader to \cref{related-work} where we discuss related results.

\begin{theorem}[{\cref{comp-nuc}, \cref{three-lims}, \cref{verdier-complete-compact-dualizable}}]\label{three-lims-intro}
	Let $R$ be an even $\Eth$-ring spectrum and let $x_1,\dotsc,x_r$ be elements of non-negative even degrees.
	The ordinary and compactly generated limits are
	\[
		\cMod_R \simeq \lim \Mod_{R/(x_1^{i_1},\dotsc,x_r^{i_r})},
		\qquad \Ind(\cMod_R^\dbl) \simeq \limw \Mod_{R/(x_1^{i_1},\dotsc,x_r^{i_r})}.
	\]
	Moreover, there are $R$-linear strongly continuous fully faithful functors
	\[
		\cMod_R
		\hoook \Ind(\cMod_R^\dbl)
		\hoook \NucR,
	\]
	and their Verdier quotients are $R/(x_1,\dotsc,x_r)$-acyclic over $R$.
\end{theorem}

We use this result to compare their K-theories after chromatic localization.
Recall that since $R$ is even, it is complex orientable, and thus has a height $n$ ideal $(p,v_1,\dotsc,v_{n-1}) \subset \pi_*(R)$.
Assuming that the elements $x_1,\dotsc,x_r$ are in (the radical of) this ideal, we deduce from \cref{three-lims-intro} that the Verdier quotients are $\Lnmf$-local.
Then, using the redshift bound proved by Clausen--Mathew--Naumann--Noel \cite{DescVan} (see \cite{redshiftQL} for an alternative proof by the author) and the purity theorem of Land--Mathew--Meier--Tamme \cite{purity}, combined with \cref{k-nuc-intro}, we get the following main result.

\begin{theorem}[{\cref{main-thm-rcg}}]\label{main-thm-rcg-intro}
	Let $R$ be an even $\Eth$-ring spectrum and let $x_1,\dotsc,x_r$ be in the radical of $(p,v_1,\dotsc,v_{n-1}) \subset \pi_*(R)$.
	The following maps are $\Tm$-local isomorphisms for $m \geq n{+}1$
	% https://q.uiver.app/#q=WzAsNixbMCwwLCJcXEtLKFIpIl0sWzEsMCwiXFxLSyhcXHhje1J9KSJdLFsyLDAsIlxcS0soXFxjTW9kX1JeXFxkYmwpIl0sWzIsMSwiXFxLSyhcXGNNb2RfUl5cXG9tZWdhKSJdLFs0LDAsIlxcbGltIFxcS0soUi8oeF8xXntpXzF9LFxcZG90c2MseF9yXntpX3J9KSkiXSxbMywwLCJcXEtjb250KFxcTnVjUikiXSxbMCwxXSxbMSwyXSxbMywyXSxbMiw1XSxbNSw0LCJcXHNpbSJdXQ==&macro_url=https%3A%2F%2Fgist.githubusercontent.com%2Fshaybenmoshe%2F301102e6fdd6d215848c519c149abcab%2Fraw%2Fquiver
	\[\begin{tikzcd}
		{\KK(R)} & {\KK(\xc{R})} & {\KK(\cMod_R^\dbl)} & {\Kcont(\NucR)} & {\lim \KK(R/(x_1^{i_1},\dotsc,x_r^{i_r}))} \\
		&& {\KK(\cMod_R^\omega)}
		\arrow[from=1-1, to=1-2]
		\arrow[from=1-2, to=1-3]
		\arrow[from=1-3, to=1-4]
		\arrow["\sim", from=1-4, to=1-5]
		\arrow[from=2-3, to=1-3]
	\end{tikzcd}\]
\end{theorem}

It is easy to see that the transition maps in the quotient tower are nilpotent extensions, thus the Dundas--Goodwillie--McCarthy theorem \cite{DGM} together with \cite{purity} give the following consequence.

\begin{cor}[{\cref{ring-dgm}}]\label{ring-dgm-intro}
	Assume moreover that $R$ is connective, then the following map is a $\Tm$-local isomorphism for $m \geq \max\{ n{+}1, 2\}$
	\[
		\KK(R) \too \lim \TC(R/(x_1^{i_1},\dotsc,x_r^{i_r})).
	\]
\end{cor}

The following is the main example of \cref{main-thm-rcg-intro} the author had in mind.

\begin{example}\label{main-example}
	Take $R$ to be any of the variants of Morava E-theory $\En$, $\JWn$, $\cJWn$ or $\BPn$ (which indeed have $\Eth$- or even $\bbE_\infty$-ring spectra structure by \cite{GH04,GH05,cJWEinf,Ell2,HW}), and take the elements $p,\dotsc,v_{n-1}$.
	\cref{main-thm-rcg-intro} applies at height $m=n{+}1$,\footnote{Note that all higher heights vanish by the redshift bound.} giving a limit over nilpotent extensions of ($2$-periodic or connective) Morava K-theory $\Ktn$, $\Kn$ or $\kn$.

	We remark that the algebraic K-theory of Morava K-theory is, at least in some sense, smaller and more accessible than that of Morava E-theory, see for example \cite[Remark 1.2.5]{syntomicKn} and \cite{K-Kn}.
\end{example}

In a different direction, as in \cref{k-nuc-intro}, \cref{main-thm-rcg-intro} holds more generally after tensoring with any dualizable category $\DD \in \PrLstdbl$.
Taking $\DD$ to be the category of sheaves or local systems, we get the following examples (in which one can replace $\Mod_R$ by any of the other module categories considered above).

\begin{example}[{\cref{sh-comp-R}}]
	For a locally compact Hausdorff space $Y$, the following map is a $\Tm$-local isomorphism
	\[
		\Kcont(\Shv(Y, \Mod_R)) \too \lim \Kcont(\Shv(Y, \Mod_{R/(x_1^{i_1},\dotsc,x_r^{i_r})})).
	\]
	By \cite[Theorem 0.2]{EfimovK}, this map identifies with
	\[
		\Gammac(Y, \KK(R)) \too \lim \Gammac(Y, \KK(R/(x_1^{i_1},\dotsc,x_r^{i_r}))).
	\]
\end{example}

\begin{example}[{\cref{fun-comp-R}}]
	For a category $\CC$, the following map is a $\Tm$-local isomorphism
	\[
		\Kcont(\Fun(\CC, \Mod_R)) \too \lim \Kcont(\Fun(\CC, \Mod_{R/(x_1^{i_1},\dotsc,x_r^{i_r})})).
	\]
\end{example}

\subsectionnotoc{Outline of the proof}\label{outline}

We now outline the proof of the main results.
We wish to highlight ahead that a single pro-constancy result is used three times, once in the proof of \cref{k-nuc-intro} and twice in the proof of \cref{three-lims-intro}.
This makes the paper somewhat non-linear, and we hope that this outline will help the reader navigate it.

\paragraph{Efimov's argument.}
Before explaining the present paper, we give some more details on Efimov's proof of \cref{k-nuc-intro} for ordinary commutative rings from \cite{EfimovLimit}.
Efimov in fact proves a much more general continuity theorem for the K-theory of dualizable categories.
He shows that for a pretower of dualizable categories $(\CC_\bullet) = (\cdots \to \CC_1 \to \CC_0)$, the assembly map
\[
	\Kcont(\limdual \CC_i) \too \lim \Kcont(\CC_i)
\]
is an isomorphism if the transition maps in the pretower satisfy certain pro-constancy and adjointability assumptions, called the strong Mittag-Leffler condition.
When the $\CC_i$ are categories of modules over rings, this translates, by Morita theory, to pro-constancy and dualizability conditions on the associated bimodules.
Efimov then proves these conditions for the universal quotient pretower $(\ZZ[x]/x^\bullet)$ by a direct computation, which implies the case of general quotients of ordinary rings by base-change.

\paragraph{Pro-constancy for $(\SS[x]/x^i)$.}
With Efimov's argument in mind, we study the universal quotient pretower $(\SS[x]/x^\bullet)$.
Rather than proving the condition required for Mittag-Leffler directly, we first prove that
\[
	\SS \simeq \SS \otimes_{\SS[x]} \SS[x] \too (\SS \otimes_{\SS[x]} \SS[x]/x^i)
\]
exhibits the pretower in the target as pro-constant (\cref{double-quotient-pro-sphere}).
Importantly, we prove this in $\Eo$-algebras.\footnote{A point that we find easy to miss is that given a conservative functor, the induced functor on pro-categories is not necessarily conservative. Therefore, we need to prove the pro-isomorphism with all the structure needed for the subsequent arguments.}
This rests on the observation that the target is isomorphic, as a graded spectrum, to $\SS \oplus \Sigma^{i|x|+1} \SS(i)$, which forces the transition maps to be trivial by weight considerations.
Next, we explain the three applications of this pro-constancy result.

\paragraph{Mittag-Leffler and \cref{k-nuc-intro}.}
Recall that the pro-constancy above is of $\Eo$-algebras.
By tensoring \emph{over} it (and a weight filtration induction), we deduce the pro-constancy of the bimodules (\cref{higher-quotient-base-Sx}) required for the Mittag-Leffler condition (\cref{ML}).
As in Efimov's case, this implies the Mittag-Leffler condition for general quotients of even $\Et$-algebras by base-change (\cref{cat-ML-quotient} and \cref{ML-quotient-Sp}).
This in turn proves \cref{k-nuc-intro} (\cref{k-nuc}) by Efimov's general continuity theorem.

\paragraph{Mittag-Leffler and the ordinary limit.}
We introduce a variant of the Mittag-Leffler condition for a tower, i.e.\ a pretower with a cone $\CC \to (\CC_\bullet)$ (\cref{bML-cond}, in fact, we observe that both Mittag-Leffler conditions make sense in a general $2$-category, which we find both conceptually clarifying and technically convenient).
We prove that under this assumption, the functor to the ordinary limit $\CC \to \lim \CC_i$ is a reflective localization (\cref{bML-ordinary-limit}).
The pro-constancy argument above proves this variant of Mittag-Leffler, exhibiting the ordinary limit of $\Mod_{R/(x_1^{i_1},\dotsc,x_r^{i_r})}$ as the subcategory of $\Mod_R$ on the complete modules (\cref{ordinary-limit}).

\paragraph{The ordinary and dualizable limits.}
The Mittag-Leffler condition gives a fully faithful embedding of the ordinary limit into the dualizable limit \cite[Remark 5.15(i)]{EfimovLimit}, and we consider its Verdier quotient.
Observe that tensoring with $\Ax := R/(x_1,\dotsc,x_r)$ is the same as extending-then-restricting-scalars along
\[
	(R/(x_1^{i_1},\dotsc,x_r^{i_r})) \ \too\ (\Ax \otimes_R R/(x_1^{i_1},\dotsc,x_r^{i_r})).
\]
The pro-constancy result above (not in the Mittag-Leffler form) says that the right-hand diagram is pro-constant on $\Ax$.
Thus, the ordinary and dualizable limits of its module categories are both $\Mod_\Ax$, and in particular their Verdier quotient vanishes.
This shows that the tensor by $\Ax$ on the Verdier quotient of the original diagram factors through zero, i.e.\ it is $\Ax$-acyclic (\cref{comp-nuc}).

We highlight a subtlety in this argument.
When $R$ is only assumed to be an $\Et$-ring spectrum, the category of \emph{left} modules over an $\Eo$-$R$-algebra is a \emph{right} module over $\Mod_R$, hence the same is true for the ordinary and dualizable limits.
On the other hand, extending-then-restricting-scalars tensors from the left at each level, which is not a priori related to the right $\Mod_R$-action.
To overcome this, we assume that $R$ is an $\Eth$-ring spectrum.
We do not know if this is necessary.

\paragraph{The compactly generated limit.}
The computation of the compactly generated limit follows from the ordinary limit (\cref{lim-w}), by noticing that modules over quotients $R/(x_1^{i_1},\dotsc,x_r^{i_r})$ are automatically complete (\cref{auto-complete}, and if $R$ is an $\Eth$-ring spectrum the same is true for right modules), which also means that dualizable and compact coincide.
The remaining Verdier quotient is then a standard compact-tensor-dualizable argument (\cref{verdier-complete-compact-dualizable}).

\paragraph{Redshift and \cref{main-thm-rcg-intro}.}
Assuming $x_1,\dotsc,x_r$ are in the radical of the height $n$ ideal, we see that $\SS/(p^{e_0},v_1^{e_1},\dotsc,v_{n-1}^{e_{n-1}}) \otimes R$ is in the thick subcategory generated by $R/(x_1,\dotsc,x_r)$ (\cref{moore-Ax}).
This means that the Verdier quotients from \cref{three-lims-intro} are $\Lnmf$-local.
\cref{main-thm-rcg-intro} then follows from the redshift bound \cite{DescVan} and purity \cite{purity} (which we slightly extend to dualizable categories in \cref{subsec-redshift}).

\subsectionnotoc{Relation to other work and further directions}\label{related-work}

\paragraph{Burklund's quotients.}
In his remarkable paper \cite{BurklundMoore}, Burklund constructs multiplicative structures on quotients of a much broader class of ring spectra (and algebras in more general stable categories), whose K-theory would be interesting to explore.
In particular, Burklund constructs $\bbE_m$-algebra structures on type $n$ generalized Moore spectra, which were subsequently refined to towers by Li--Zhang \cite[Proposition 2.1.4]{LiZhang} and Meyer--Wagner \cite[Example 2.34]{refinedTC}.
It seems very likely that the results of this paper hold for these quotients as well.
Indeed, Meyer--Wagner's \cite[2.20(V)]{refinedTC} is very close to our pro-constancy result, and Li--Zhang computed the ordinary limit, which we now discuss.

\paragraph{Modules over formal completions.}
For ordinary commutative rings, and more generally for animated rings, the statement that the ordinary limit is equivalent to complete modules is well-known, see for example \cite[\href{https://stacks.math.columbia.edu/tag/0H0L}{Tag 0H0L}]{stacks-project} and \cite[Lemma 2.2.2]{HLP}.
In the setting of adic connective $\bbE_\infty$-ring spectra (and more general formal spectral stacks), a closely related statement is proved in \cite[Theorem 8.3.4.4]{SAG} (and earlier in \cite[Theorem 5.1.9]{DAG-XII}), but only for the (almost) connective modules rather than all modules (see \cite[Remark 8.3.4.5]{SAG}).
Nonetheless, Li--Zhang prove that for a $\Kn$-local $\bbE_\infty$-ring spectrum, the limit of modules over the tensor with the tower of generalized Moore spectra is the $\Kn$-local $R$-modules \cite[Main Theorem A]{LiZhang}.

In a different direction, \cite[Corollary 11.4]{sep-gal} shows that in the setting of adic connective $\bbE_\infty$-ring spectra, the equivalence does hold for dualizable modules (with no connectivity assumptions).
Upon taking Ind, this proves the analogous result for the compactly generated limit.

\paragraph{Picard groups of quotients.}
Li--Zhang \cite{LiZhang} and Levy--Li--Zhang \cite{picquot} study the Picard groups of quotients of ring spectra, both in the context of even ring spectra as in this paper and for the generalized Moore spectra mentioned above.
They prove several structural results, as well as carry out explicit computations, especially at height $1$.
In light of the map from Picard to the units of K-theory, it would be interesting to explore whether their results can shed light on the K-theory of quotients.
See also \cite{levy-k1,lee-levy} for the state of the art on the algebraic K-theory of the $\Ko$-local sphere.

\paragraph{Quotients in other stable categories.}
In the body of the paper we work in a slightly more general context, allowing a general presentably monoidal stable category $\EE$ (the case considered in the introduction being $\Mod_R$), see \assref{setup}.
One could, for example, consider genuine equivariant spectra, which may be of interest especially in light of recent developments \cite{DescVan,EH,HilmanMotivic,CHLL,HR}, though the author does not have any applications in mind.

\paragraph{Crystallinity of algebraic K-theory.}
A striking result of Hahn--Levy--Senger \cite{crystallinity} shows that for $R$ satisfying height $n$ Lichtenbaum--Quillen, such as $\BPn$ \cite{HW}, the algebraic K-theory $\KK(R)$ modulo powers of $p,\dotsc,v_{n+1}$ depends only on the reduction of $R$ modulo sufficiently high powers of $p,\dotsc,v_n$.
The relation to our result is somewhat indirect.
Beyond the Lichtenbaum--Quillen hypothesis, their result is about dependence on a finite stage, whereas ours is about isomorphism at the limit.
More importantly, both their input and output quotients are shifted up by one generator relative to ours (see \cref{main-example}).

\paragraph{Ordinary to dualizable quotient.}
\cref{three-lims-intro} gives a qualitative statement about the Verdier quotient of the ordinary limit into the dualizable limit, and it is natural to ask for a more precise description.
For ordinary commutative rings, or schemes more generally, such results are given for the Clausen--Scholze version of nuclear modules in \cite[\S10.3, see in particular Equation (10.13)]{rigidity}, and announced for the dualizable limit version in the subsequent subsection of \emph{loc.\ cit.}

\subsectionnotoc{Notation and conventions}

\begin{enumerate}
	\item The space or spectrum of maps between $X, Y \in \CC$ is denoted $\hom_\CC(X,Y)$ or simply $\hom(X,Y)$.
	\item By default, a module means a \emph{left} module, and their category is denoted $\Mod_A(\CC)$ or simply $\Mod_A$.
	\item By default, a dualizable module means a \emph{left} dualizable \emph{left} module in the sense of \cite[\S4.6.2]{HA} (and we remind the reader that a left module is \emph{right} dualizable if and only if its underlying object is right dualizable by \cite[Proposition 4.6.2.13]{HA}).
	\item For $\EE \in \Alg(\PrL)$, we denote by $\PrLE := \RMod_\EE(\PrL)$ the \emph{right} modules over $\EE$, taking the opposite convention to \cite{EfimovLimit} (this is motivated by the fact that $\Mod_R(\EE)$ is a right $\EE$-module).
	\item We denote by $\PrLstw \subset \PrLstdbl \subset \PrL$ the (non-full) subcategories on the compactly generated or dualizable presentable stable categories and strongly continuous functors between them.
	\item We denote the natural numbers with decreasing morphisms by $\Ng := \{ \cdots \to 1 \to 0 \}$.
	\item We use the term \emph{pretower} for a diagram $(X_\bullet) = (\cdots \to X_1 \to X_0)$, and the term \emph{tower} for a pretower with a cone $X \to (X_\bullet)$.
	\item We denote by $c\colon \EE \to \Pro(\EE)$ the constant pro-object functor.
\end{enumerate}

\subsectionnotoc{Acknowledgements}

I thank Qingyuan Bai and Wyatt Reeves for helpful conversations and suggestions, and Ishan Levy for useful email correspondence.
I also thank the speakers Alexander Efimov and Akhil Mathew, and the organizers Tobias Barthel, Kaif Hilman, Dominik Kirstein, and Jonas McCandless, for the 2024 workshop on ``Dualisable Categories \& Continuous K-theory'' in the Max Planck Institute for Mathematics.
This work was supported by the Max Planck--Weizmann joint postdoctoral program.

% I thank the Herbert Foundation for kindly granting permission to reproduce the image of Sol LeWitt's \emph{Incomplete Open Cubes} on the first page of this manuscript.

I used AI tools, primarily ChatGPT 5.5 and 5.6, frequently throughout the course of this project.
It was used as a research assistant on various tasks, ranging from developing and extending arguments together, through locating references and explaining papers, to checking the completeness of proofs.
The manuscript was written by myself.

	\section{Mittag-Leffler conditions and limits}

\subsection{Tower and pretower Mittag-Leffler conditions}\label{subsec-ml}

To start, we recall the definition of the strong Mittag-Leffler condition from \cite[Definition 5.1]{EfimovLimit}.
In fact, we make the small observation that the condition (aside from the dualizability assumption) can be formulated in a general $2$-category.

\begin{defn}\label{ML-cond}
	Let $\bbA$ be a $2$-category.
    Let $(\CC_\bullet)$ be a pretower in $\bbA$ with left adjoint transition maps $F_{ij}\colon \CC_i \to \CC_j$.
    We say that $(\CC_\bullet)$ satisfies the \tdef{pretower Mittag-Leffler condition} if for every $j$ and $k$
	\begin{enumerate}
		\item the following pro-object is pro-constant
		\[
			\prolim[i \geq j,k] F_{ij} F_{ik}^\RR
			\qin \Pro(\hom_{\bbA}(\CC_k, \CC_j)),
		\]
		\item and its limit,\footnote{This limit exists since the pro-object corresponding to the diagram is, by assumption, pro-constant.} that is the morphism below, admits both a left and a right adjoint in $\bbA$
		\[
			(\lim_{i \geq j,k} F_{ij} F_{ik}^\RR)\colon \CC_k \too \CC_j.
		\]
	\end{enumerate}
\end{defn}

\begin{remark}\label{enough-j-k}
	It suffices to check the first condition for $j=k$.
	Indeed, if $j > k$, then $F_{ik} \simeq F_{jk} F_{ij}$ so that the $j,k$ case is obtained from the $j,j$ case by pre-composing with $F_{jk}^\RR$, and similarly for $k > j$.
\end{remark}

We now verify that this is indeed equivalent to the strong Mittag-Leffler condition from \cite[Definition 5.1]{EfimovLimit}.

\begin{prop}\label{ML-equiv}
	Let $\EE \in \Alg(\PrLst)$ be rigid.
	A pretower $(\CC_\bullet) \in \PrLE$ of \emph{dualizable} $\EE$-linear categories satisfies the strong Mittag-Leffler condition over $\EE$, if and only if it satisfies the pretower Mittag-Leffler condition in the $2$-category $\PrLE$.
\end{prop}

\begin{proof}
	We need to check that the two conditions in \cref{ML-cond} are equivalent to the two conditions in \cite[Definition 5.1]{EfimovLimit}.
	The first condition in \emph{loc.\ cit.} is the special case of the first condition in \cref{ML-cond} for $j=k$, and the two are equivalent by \cref{enough-j-k}.
	The second condition in \cref{ML-cond} requires their existence in $\PrLE$, while the second condition in \emph{loc.\ cit.} requires their existence only in $\PrLst$.
	Since $\EE$ is rigid, $\EE$-linearity is automatic by \cite[Lemma 9.3.6]{GR}.
\end{proof}

We now introduce a stronger condition, done relative to a cone over the pretower.
In the next subsection, we will see that this variant is useful for computing the limit of the pretower, see \cref{bML-ordinary-limit}.

\begin{defn}\label{bML-cond}
	Let $\bbA$ be a $2$-category.
    Let $\CC \to (\CC_\bullet)$ be a tower in $\bbA$ such that the maps $u_i\colon \CC \to \CC_i$ and $F_{ij}\colon \CC_i \to \CC_j$ are left adjoints.
    We say that $(\CC_\bullet)$ satisfies the \tdef{tower Mittag-Leffler condition} if
	\begin{enumerate}
		\item for every $j$ and $k$, the counits of the $u_i \dashv u_i^\RR$ adjunctions induce a pro-isomorphism
		\[
            c(u_j u_k^\RR)
            \simeq \prolim[i \geq j,k] F_{ij} u_i u_i^\RR F_{ik}^\RR
            \too \prolim[i \geq j,k] F_{ij} F_{ik}^\RR
			\qin \Pro(\hom_{\bbA}(\CC_k, \CC_j)),
		\]
		\item and each $u_i$ admits a left adjoint, and its right adjoint $u_i^\RR$ admits a further right adjoint.
	\end{enumerate}
\end{defn}

We observe that this indeed implies the pretower Mittag-Leffler.

\begin{prop}\label{bML-cML}
    Let $\bbA$ be a $2$-category, and let $\CC \to (\CC_\bullet)$ be a Mittag-Leffler tower, then $(\CC_\bullet)$ is Mittag-Leffler pretower.
\end{prop}

\begin{proof}
    The first condition in \cref{bML-cond} identifies the pro-object in the first condition of \cref{ML-cond} as the constant pro-object on $u_j u_k^\RR$.
    The limit from the second condition of \cref{ML-cond} is then $u_j u_k^\RR$, which admits both a left and a right adjoint by the second condition of \cref{bML-cond}.
\end{proof}

One advantage of the $2$-categorical formulation is that it is manifestly preserved by $2$-functors.

\begin{prop}\label{ML-2}
	Any $2$-functor $\bbA \to \bbB$ preserves Mittag-Leffler towers and pretowers.
\end{prop}

\begin{proof}
	A $2$-functor preserves adjunctions and pro-objects, and in particular pro-constancy.
    For the second condition of \cref{ML-cond}, we note that the limit in question is the value of a pro-constant object, hence is preserved by the $2$-functor.
\end{proof}

A particularly useful example for this (which was used in \cite{EfimovLimit}) is the following passage from the Morita $2$-category.
Recall that for $\EE \in \Alg(\PrL)$, the \tdef{Morita $2$-category} $\mdef{\MorE}$, is the $2$-category whose objects are $\Eo$-algebras in $\EE$, and the hom category from $R$ to $R'$ is the category $\BMod{R'}{R}(\EE)$, as defined in \cite[Definition 4.40]{morita} (see also \cite[Remark 4.8.4.9]{HA}, \cite[Proposition 5.13]{morita2}).
Note that it receives a functor $\Alg(\EE) \to \MorE$ sending $R$ to $R$ and $f\colon R \to R'$ to the $1$-morphism given by the $R'$-$R$-bimodule $R'$, which has right adjoint given by the $R$-$R'$-bimodule $R'$ (corresponding to extension and restriction of scalars along $f$ on the categories of modules \cite[Theorem 4.8.4.1]{HA}).
With this in place, we now state the pretower Mittag-Leffler condition in the Morita $2$-category.

\begin{prop}\label{cat-ML-morita}
	Let $\EE \in \Alg(\PrL)$ and let $(R_\bullet)$ be a pretower in $\Alg(\EE)$.
    Its image in $\MorE$ satisfies the pretower Mittag-Leffler condition if and only if for every $j$ and $k$
	\begin{enumerate}
		\item the following pro-bimodule is pro-constant
		\[
			\prolim[i \geq j,k] (R_j \otimes_{R_i} R_k)
			\qin \Pro(\BMod{R_j}{R_k}(\EE)),
		\]
		\item and its limit, i.e.\ the $R_j$-$R_k$-bimodule $\lim_{i \geq j,k} (R_j \otimes_{R_i} R_k)$, is left dualizable as a left $R_j$-module and right dualizable as a right $R_k$-module.
	\end{enumerate}
    Similarly, the tower $\one \to (R_\bullet)$ satisfies the tower Mittag-Leffler condition in $\MorE$ if and only if
    \begin{enumerate}
        \item for every $j$ and $k$ the following map of pro-bimodules is a pro-isomorphism
        \[
            c(R_j \otimes R_k)
            \too \prolim[i \geq j,k] (R_j \otimes_{R_i} R_k)
            \qin \Pro(\BMod{R_j}{R_k}(\EE)),
        \] 
        \item and each $R_i \in \EE$ is left and right dualizable.
    \end{enumerate}
\end{prop}

\begin{proof}
    We prove the first statement, the second is essentially the same.
	The description of the image of $\Alg(\EE)$ in $\MorE$ shows that in this case $F_{ij} F_{ik}^\RR$ is indeed the $R_j$-$R_k$-bimodule $R_j \otimes_{R_i} R_k$, which identifies the first condition.
	For the second condition, adjointability of $1$-morphisms in $\MorE$ is indeed equivalent to dualizability of the corresponding bimodule by \cite[Proposition 4.6.2.10 together with Proposition 4.6.2.13]{HA}.
\end{proof}

This gives a description of the Mittag-Leffler conditions for module categories.

\begin{cor}\label{ML-ring-to-mod}
	Let $\EE \in \Alg(\PrL)$.
    If $(R_\bullet) \in \Alg(\EE)$ satisfies the pretower Mittag-Leffler condition in $\MorE$, then $(\Mod_{R_\bullet}(\EE))$ satisfies the pretower Mittag-Leffler condition in $\PrLE$, and similarly for towers.

    If moreover $\EE \in \Alg(\PrLst)$ is rigid, then $(\Mod_{R_\bullet}(\EE))$ satisfies the strong Mittag-Leffler condition over $\EE$.
\end{cor}

\begin{proof}
	By \cref{ML-2}, the $2$-functor $\MorE \to \PrLE$ sending $R$ to $\Mod_R(\EE)$ preserves the tower and pretower Mittag-Leffler conditions, which proves the first part.
    For the second part, note that each $\Mod_{R_i}(\EE)$ is dualizable in $\PrLE$ by \cite[Remark 4.8.4.8]{HA}, and the result follows using \cref{ML-equiv}.
\end{proof}

\subsection{Ordinary limit of Mittag-Leffler towers and completion}\label{subsec-ml-limit}

In this subsection we identify the ordinary limit.
Let $\EE \in \Alg(\PrL)$.
Let $\CC \to (\CC_\bullet)$ be a Mittag-Leffler tower in $\PrLE$.
Being a cone, the functors $u_i$ assemble to a map into the ordinary limit
\[
    u\colon \CC \too \lim_i \CC_i,
    \qquad u(M) = (u_\bullet(M)).
\]
By assumption each $u_i$ has a right adjoint $u_i^\RR$, therefore, by \cite[Theorem B]{adjdesc}, the functor $u$ has a right adjoint given by their limit
\[
    u^\RR\colon \lim_i \CC_i \too \CC,
    \qquad
    u^\RR((M_\bullet)) = \lim_i u_i^\RR(M_i).
\]
We introduce the following definition.

\begin{defn}
    Let $\CC \to (\CC_\bullet)$ be a Mittag-Leffler tower in $\PrLE$.
    We say that $M \in \CC$ is \tdef{complete} if the unit of the $u \dashv u^\RR$ adjunction is an isomorphism on it, namely if
    \[
        M \iso u^\RR(u(M)) \simeq \lim_i u_i^\RR(u_i(M))
        \qin \CC,
    \]
    and we denote their full subcategory by $\mdef{\CCc} \subset \CC$.
\end{defn}

\begin{thm}\label{bML-ordinary-limit}
    Let $\CC \to (\CC_\bullet)$ be a Mittag-Leffler tower in $\PrLE$.
    The right adjoint $u^\RR$ is fully faithful and $u$ restricts to an equivalence
    \[
        u\colon \CCc \iso \lim_i \CC_i.
    \]
    In particular, $\CCc$ is presentable and a reflective localization of $\CC$.
\end{thm}

Before proving this, we need the following observation about pro-objects.

\begin{lem}\label{pro-pro}
	Let $I_\bullet\colon K \to \Cat$ be a diagram, denote by $I \to K$ its unstraightening, and assume that each $I_k$ is cofiltered and so is $I$.
	Let $\DD$ be a category, and let $X_\bullet \to Y_\bullet$ be a morphism in $\Fun(I, \DD)$.
	Assume that for any fixed $k \in K$, the map of pro-objects restricted to $I_k$
	\[
		\prolim[i \in I_k] X_{k,i} \iso \prolim[i \in I_k] Y_{k,i}
		\qin \Pro(\DD)
	\]
	is an isomorphism.
	Then the map of pro-objects
	\[
		\prolim[(k,i) \in I] X_{k,i} \too \prolim[(k,i) \in I] Y_{k,i}
		\qin \Pro(\DD)
	\]
	is an isomorphism.
\end{lem}

\begin{proof}
	Let $Z$ be an object of $\DD$.
	Using the formula for the hom space in pro-categories, we get
	\begin{align*}
		\hom(\prolim[i \in I] Y_i, c(Z))
		&\simeq \colim_{i \in I^\op} \hom(Y_i, Z)\\
		&\simeq \colim_{k \in K^\op} \colim_{i \in I_k^\op} \hom(Y_i, Z)\\
		&\simeq \colim_{k \in K^\op} \colim_{i \in I_k^\op} \hom(X_i, Z)\\
		&\simeq \colim_{i \in I^\op} \hom(X_i, Z)\\
		&\simeq \hom(\prolim[i \in I] X_i, c(Z)),
	\end{align*}
	where the third isomorphism follows from the assumption that the map of pro-objects restricted to $I_k$ is an isomorphism for any fixed $k \in K$.
	The formula for the hom space in pro-categories shows this is also true for any pro-object in place of $c(Z)$, and so we conclude by the Yoneda lemma.
\end{proof}

We now return to the proof of the identification of the ordinary limit.

\begin{proof}[{Proof of \cref{bML-ordinary-limit}}]
    We need to show that the right adjoint
    \[
        u^\RR\colon \lim_i \CC_i \too \CC
    \]
    is fully faithful.
    Equivalently, we need to show that the counit
    \[
        u u^\RR \too \Id_{\lim_i \CC_i}
    \]
    is an isomorphism.
    Since this is a natural transformation, it suffices to check after evaluating at every object $(M_\bullet) \in \lim_i \CC_i$.
    Since the projections $\lim_i \CC_i \to \CC_j$ are jointly conservative, it suffices to check after each projection.
    That is, we need to show that the map
    \[
        u_j(\lim_i u_i^\RR(M_i)) \simeq u_j(u^\RR((M_\bullet))) \too M_j
        \qin \CC_j
    \]
    is an isomorphism for every $j$ and $(M_\bullet)$.

    Consider the maps
    \[
        u_j(u_k^\RR(M_k)) \too F_{ij}(F_{ik}^\RR(M_k))
        \qin \CC_j,
    \]
    which assembles into a diagram indexed by $A = \{ (k, i) \mid i \geq j,k \}$.
    By the first part of the tower Mittag-Leffler condition, for every fixed $k$, the associate map of pro-objects over $i \geq j,k$ is a pro-isomorphism
    \[
        \prolim[i \geq j,k] u_j(u_k^\RR(M_k))
        \iso \prolim[i \geq j,k] F_{ij}(F_{ik}^\RR(M_k))
        \qin \Pro(\CC_j).
    \]
    Since this holds for every fixed $k$, by \cref{pro-pro} the original diagram induces a pro-isomorphism
    \[
        \prolim[(k,i) \in A] u_j(u_k^\RR(M_k))
        \iso \prolim[(k,i) \in A] F_{ij}(F_{ik}^\RR(M_k))
        \qin \Pro(\CC_j).
    \]
    Taking the limit, we get an isomorphism
    \[
        \lim_{(k,i) \in A} u_j(u_k^\RR(M_k))
        \iso \lim_{(k,i) \in A} F_{ij}(F_{ik}^\RR(M_k))
        \qin \CC_j.
    \]
    Observe that the diagonal $\{ (i,i) \mid i \geq j \} \hookrightarrow A$ is initial, so we get an isomorphism
    \[
        \lim_{i \geq j} u_j(u_i^\RR(M_i))
        \iso \lim_{i \geq j} F_{ij}(F_{ii}^\RR(M_i))
        \qin \CC_j.
    \]
    Now, we note that $F_{ii}^\RR$ is the identity, and $F_{ij}(M_i) \simeq M_j$, so the right hand side is the constant limit of $M_j$.
    For the source, the second part of the tower Mittag-Leffler condition says that $u_j$ is a right adjoint, and so it commutes with limits.
    Together, we get the required isomorphism
    \[
        u_j(u^\RR((M_\bullet))) \simeq u_j(\lim_i u_i^\RR(M_i)) \iso M_j
        \qin \CC_j.
    \]
\end{proof}

\begin{defn}\label{u-acyclic-local}
    Let $\EE \in \Alg(\PrLst)$, and let $\CC \to (\CC_\bullet)$ be a tower in $\PrLE$.
    We define the \tdef{$u$-acyclics} to be $\ker(u) \subset \CC$, that is, the objects $N \in \CC$ such that $u(N) = 0$.
    We define the \tdef{$u$-local} to be the right orthogonal $(\ker(u))^\perp \subset \CC$, that is, the objects $M \in \CC$ such that $\hom(N,M) = 0$ for every $N \in \ker(u)$.
\end{defn}

\begin{prop}\label{complete-local}
    Let $\EE \in \Alg(\PrLst)$, and let $\CC \to (\CC_\bullet)$ be a Mittag-Leffler tower in $\PrLE$.
    Then $X \in \CC$ is complete if and only if it is $u$-local, that is
    \[
        \CCc \ =\ (\ker(u))^\perp.
    \]
\end{prop}

\begin{proof}
    \cref{bML-ordinary-limit} shows that $u\colon \CC \to \CCc$ is a reflective localization.
    This is then a standard consequence when the categories are stable, c.f.\ \cite[Lemma 4.1.8]{AmbiHeight}.
\end{proof}

\begin{prop}\label{complete-local-conservative}
    Let $\EE \in \Alg(\PrLst)$, and let $\CC \to (\CC_\bullet)$ be a Mittag-Leffler tower in $\PrLE$.
    If all $F_{i0}$ are conservative, then $N \in \CC$ is $u$-acyclic if and only if $u_0(N) = 0$, hence
    \[
        \CCc \ =\ (\ker(u_0))^\perp.
    \]
\end{prop}

\begin{proof}
    This is an immediate consequence of \cref{complete-local}, since $F_{i0} u_i \simeq u_0$ and $F_{i0}$ is conservative.
\end{proof}

	\section{The graded quotient tower}\label{sec-gr}
\subsection{Graded polynomial algebras and their quotients}\label{subsec-gr-poly}

We begin with a quick survey of graded spectra, the graded polynomial algebra, and its quotients.
We primarily follow the weight truncation formalism of \cite[Appendix B]{HW}, and refer the reader to \cite{rotation,E2quot,picquot,ABM,RSS} for related constructions and results.

\begin{defn}
	We denote the presentably symmetric monoidal categories of \tdef{graded spectra}, endowed with Day convolution, by
	\[
		\mdef{\GrZ} := \Fun(\ZZ, \Sp)
		\qin \CAlg(\PrLst).
	\]
	For a graded spectrum $X \in \GrZ$, we denote its weight $m$ component by $X_m \in \Sp$.
\end{defn}

\begin{defn}
	For $i \in \ZZ$, we denote by $\mdef{\SS(i)}$ the sphere spectrum placed in weight $i$.
	For a graded spectrum $X \in \GrZ$ we define its \tdef{$i$-fold weight shift} to be
	\[
		\mdef{X(i)} := \SS(i) \otimes X \qin \GrZ.
	\]
\end{defn}

\begin{remark}\label{shift-dual}
	The name weight shift is justified by the following standard observation
	\[
		(X(i))_m
		= (\SS(i) \otimes X)_m
		\simeq \bigoplus_{a+b=m} \SS(i)_a \otimes X_b
		\simeq X_{m-i}.
	\]
	This also shows that $\SS(i)$ is dualizable, hence $X(i)$ is dualizable if (and only if) $X$ is dualizable.
\end{remark}

Later in the paper we shall be interested in realizing a graded spectrum into a spectrum.
Observe that the functor $\ZZ \to \pt$ is symmetric monoidal, hence left Kan extension along it is symmetric monoidal for Day convolution (c.f.\ \cite[Proposition 3.6]{BMS}), which we make into the following definition.

\begin{defn}\label{realization}
	We define the symmetric monoidal \tdef{realization} functor to be the left Kan extension
	\[
		\GrZ \too \Sp,
		\qquad X \mapstoo \bigoplus_{m \in \ZZ} X_m.
	\]
\end{defn}

We shall also be interested in subcategories of graded spectra supported on a collection of weights.
While the inclusion has an adjoint from both sides, which are in fact the same, given by projecting onto the corresponding weights, these are not generally compatible with the symmetric monoidal structure.
However, in the following special cases, we do have compatibility with the symmetric monoidal structure.

\begin{defn}
	We denote the full subcategory of graded spectra supported on non-negative weights, and the further full subcategory on those supported on weights $\{0, \dots, i{-}1\}$, by
	\[
		\mdef{\Gri} \quad\subset\quad \mdef{\GrN} \quad\subset\quad \GrZ.
	\]
	We define the \tdef{weight truncation} functor to be the projection
	\[
		\mdef{w_{<i}}\colon \GrN \too \Gri.
	\]
\end{defn}

The subcategory $\GrN \subset \GrZ$ is a symmetric monoidal subcategory since it is closed under the tensor product and contains the unit (see for example \cite[Lemma B.0.6(iii)]{HW}).
More importantly for our purposes, weight truncation is symmetric monoidal.

\begin{prop}[{\cite[Lemma B.0.6(iv)]{HW}}]\label{weight-trunc-tower}
	The weight truncation functor
	\[
		w_{<i}\colon \GrN \too \Gri
	\]
	is compatible with the symmetric monoidal structure.
	That is, there is a symmetric monoidal structure on $\Gri$ such that $w_{<i}$ is symmetric monoidal, hence the inclusion is lax symmetric monoidal.

	These assemble into the following tower in $\CAlg(\PrLst)$
	% https://q.uiver.app/#q=WzAsOSxbMSwwLCJcXGNkb3RzIl0sWzAsMCwiXFxHciJdLFsyLDAsIlxcR3JfezxpKzF9Il0sWzMsMCwiXFxHcl97PGl9Il0sWzQsMCwiXFxjZG90cyJdLFs1LDAsIlxcR3JfezwxfSJdLFs2LDAsIlxcR3JfezwwfSJdLFs2LDEsIjAiXSxbNSwxLCJcXFNwIl0sWzAsMiwid197PGkrMX0iXSxbMSwwXSxbMiwzLCJ3X3s8aX0iXSxbMyw0XSxbNCw1LCJ3X3s8MX0iXSxbNSw2LCJ3X3s8MH0iXSxbNSw4LCJcXHJvdGF0ZWJveHs5MH17JFxcdGV4dHN0eWxlXFxzaW1lcSR9IiwxLHsic3R5bGUiOnsiYm9keSI6eyJuYW1lIjoibm9uZSJ9LCJoZWFkIjp7Im5hbWUiOiJub25lIn19fV0sWzYsNywiXFxyb3RhdGVib3h7OTB9eyRcXHRleHRzdHlsZVxcc2ltZXEkfSIsMSx7InN0eWxlIjp7ImJvZHkiOnsibmFtZSI6Im5vbmUifSwiaGVhZCI6eyJuYW1lIjoibm9uZSJ9fX1dXQ==&macro_url=https%3A%2F%2Fgist.githubusercontent.com%2Fshaybenmoshe%2F301102e6fdd6d215848c519c149abcab%2Fraw%2Fquiver
	\[\begin{tikzcd}[row sep=tiny]
		\Gr & \cdots & {\Gr_{<i+1}} & {\Gr_{<i}} & \cdots & {\Gr_{<1}} & {\Gr_{<0}} \\
		&&&&& \Sp & 0
		\arrow[from=1-1, to=1-2]
		\arrow["{w_{<i+1}}", from=1-2, to=1-3]
		\arrow["{w_{<i}}", from=1-3, to=1-4]
		\arrow[from=1-4, to=1-5]
		\arrow["{w_{<1}}", from=1-5, to=1-6]
		\arrow["{w_{<0}}", from=1-6, to=1-7]
		\arrow["{\rotatebox{90}{$\textstyle\simeq$}}"{description}, draw=none, from=1-6, to=2-6]
		\arrow["{\rotatebox{90}{$\textstyle\simeq$}}"{description}, draw=none, from=1-7, to=2-7]
	\end{tikzcd}\]
\end{prop}

\begin{remark}
	This fails if one does not restrict to non-negative weights, which is the reason we consider the non-negative weight subcategory in this subsection.
	Indeed, the acyclics, namely, graded spectra supported on weights $\geq i$ are closed under tensoring with graded spectra supported on weights $\geq 0$, but not, for example, with $\SS(-1)$.
	Later, we pass back to all graded spectra, primarily since $\SS(1)$, and related objects, are not dualizable in $\GrN$.
\end{remark}

Consider $\Sigma^d\SS(1)$, the graded spectrum having $\Sigma^d \SS$ in weight $1$ and degree $d \in \ZZ$, and the free $\Eo$-algebra on it $\free_{\Eo}(\Sigma^d\SS(1))$.
Remarkably, if $d$ is even, this free $\Eo$-algebra has a canonical $\Et$-algebra structure via the J-homomorphism.
This was first observed in \cite{rotation}, and we refer the reader to \cite[Construction 4.4]{RSS} for another useful recent discussion.

\begin{prop}[{\cite{rotation}}]
	If $d$ is \emph{even} then $\free_{\Eo}(\Sigma^d\SS(1))$ has a canonical $\Et$-algebra structure in $\Gr$.
\end{prop}

\begin{remark}[{\cite[Remark 3.11]{disk}, \cite[Remark 4.5]{RSS}}]
	This does \emph{not} further lift to an $\Eth$-algebra.
\end{remark}

We use the following notation.

\begin{defn}
	For an even integer $|x| \in 2\ZZ$, we define the (graded) \tdef{polynomial algebra} to be the free $\Eo$-algebra on $\Sigma^{|x|}\SS(1)$, with the $\Et$-algebra structure discussed above
	\[
		\mdef{\SS[x]} := \free_{\Eo}(\Sigma^{|x|}\SS(1))
		\qin \Alg_{\Et}(\GrZ).
	\]
\end{defn}

Since the weight truncation $w_{<i}$ is symmetric monoidal and the inclusion is lax symmetric monoidal (\cref{weight-trunc-tower}), they take $\Et$-algebras to $\Et$-algebras, and we introduce the following definition.

\begin{defn}\label{trunc-poly}
	We denote the \tdef{$i$-truncated polynomial algebra}, together with the unit of the adjunction which we call the quotient map, by
	\[
		\mdef{\SS[x]/x^i} := w_{<i}(\SS[x]) \qin \Alg_{\Et}(\Gri),
		\qquad\qquad
		\SS[x] \too \SS[x]/x^i
		\qin \Alg_{\Et}(\GrZ).
	\]
	These assemble into a tower
	\[
		\SS[x] \too (\SS[x]/x^i)_{i \in \Ng}
		\qin \Alg_{\Et}(\GrZ).
	\]
\end{defn}

As the notation suggests, we have for example $\SS[x]/x^0 \simeq 0$ and $\SS[x]/x \simeq \SS$.
We now wish to show that $\SS[x]/x^i$ is, as the notation suggests, the cofiber of $\SS[x]$ by multiplication by $x^i$, as an $\SS[x]$-$\SS[x]$-bimodule, and we open with the following.

\begin{warn}
	Recall that for an $\Eo$-algebra $R$ and an element $y$, multiplication from the left and from the right give two a priori different maps $y \cdot -$ and $- \cdot y$.
	Associativity makes multiplication from the left into a right $R$-module map, and, similarly, multiplication from the right is a left $R$-module map.
	If $R$ is an $\Et$-algebra, we can identify the two maps on underlying objects.
	Moreover, we can make them into an $R$-$R$-bimodule map.
	We warn the reader however that there is no canonical choice of such a lift (while it can be chosen uniformly).
	To exemplify this, the identification
	\[
		y\cdot(r \cdot -) \simeq r \cdot (y \cdot -)
	\]
	requires a choice of braiding, which can be taken to be $\beta_{y,r}$, the inverse of $\beta_{r,y}$, or a composition thereof.
	For an extended discussion of such phenomena we refer the reader to \cite[\S2.1, Proposition 3.8]{inv-braid} and \cite[Proposition 2.35]{dual-braid}.
\end{warn}

With this in mind, we prove the following.

\begin{prop}\label{quotient-cofiber}
	For any $i \geq 0$ and any choice of $\SS[x]$-$\SS[x]$-bimodule multiplication-by-$x^i$ map, we have an exact sequence
	\[
		\Sigma^{i|x|} \SS[x](i) \too[x^i] \SS[x] \too \SS[x]/x^i
		\qin \BMod{\SS[x]}{\SS[x]}(\GrZ).
	\]
\end{prop}

\begin{proof}
	Recall from \cref{weight-trunc-tower} that weight truncation adjunction is symmetric monoidal, hence its unit preserves the $\SS[x]$-$\SS[x]$-bimodule structure.
	Observe that $\Sigma^{i|x|} \SS[x](i)$ is concentrated in weights $\geq i$, and so is sent to zero by $w_{<i}$.
	Therefore, applying the unit of the adjunction to the multiplication-by-$x^i$ map, we get the following commutative square
	% https://q.uiver.app/#q=WzAsNSxbMCwwLCJcXFNpZ21hXntpfHh8fSBcXFNTW3hdKGkpIl0sWzAsMiwiMCJdLFsxLDIsIlxcU1NbeF0veF5pIl0sWzEsMCwiXFxTU1t4XSJdLFsyLDEsIlxcQk1vZHtcXFNTW3hdfXtcXFNTW3hdfShcXEdyWikiXSxbMCwzLCJ4XmkiXSxbMCwxXSxbMywyXSxbMSwyXSxbNyw0LCJ7XFx0ZXh0c3R5bGVcXGlufSIsMSx7ImxhYmVsX3Bvc2l0aW9uIjo3MCwic2hvcnRlbiI6eyJzb3VyY2UiOjIwfSwic3R5bGUiOnsiYm9keSI6eyJuYW1lIjoibm9uZSJ9LCJoZWFkIjp7Im5hbWUiOiJub25lIn19fV1d
	\[\begin{tikzcd}[row sep=tiny]
		{\Sigma^{i|x|} \SS[x](i)} & {\SS[x]} & \\
		&& {\BMod{\SS[x]}{\SS[x]}(\GrZ)} \\
		0 & {\SS[x]/x^i}
		\arrow["{x^i}", from=1-1, to=1-2]
		\arrow[from=1-1, to=3-1]
		\arrow[""{name=0, anchor=center, inner sep=0}, from=1-2, to=3-2]
		\arrow[from=3-1, to=3-2]
		\arrow["{{\textstyle\in}}"{description, pos=0.7}, draw=none, from=0, to=2-3]
	\end{tikzcd}\]
	and it remains to show that this square is (co)cartesian.
	Since the forgetful from bimodules is conservative and exact, it suffices to check this on underlying graded spectra.
	Since (co)limits in graded spectra are computed weight-wise, we can check this in each weight separately.
	Finally, in weights $\geq i$, the multiplication map is an isomorphism while $\SS[x]/x^i$ vanishes, and in weights $< i$, the quotient map is an isomorphism while $\Sigma^{i|x|} \SS[x](i)$ vanishes, concluding the proof.
\end{proof}

\begin{prop}\label{quotient-induct}
	For any $i \geq 0$, we have an exact sequence
	\[
		\Sigma^{i|x|} \SS(i) \too \SS[x]/x^{i+1} \too \SS[x]/x^i
		\qin \BMod{\SS[x]/x^{i+1}}{\SS[x]/x^{i+1}}(\GrZ).
	\]
\end{prop}

\begin{proof}
	This follows from \cref{quotient-cofiber} by applying $w_{<i+1}$, which is exact and symmetric monoidal by \cref{weight-trunc-tower}.
\end{proof}

\begin{prop}\label{quotient-dualizable}
	For any $i \geq 0$, the left $\SS[x]$-module $\SS[x]/x^i \in \Mod_{\SS[x]}(\GrZ)$ is both left dualizable and right dualizable.
\end{prop}

\begin{proof}
	For dualizability from the left, \cref{quotient-cofiber} shows that $\SS[x]/x^i$ is the cofiber of $\SS[x]$ and $\Sigma^{i|x|} \SS[x](i)$.
	Note that $\SS$ and $\Sigma^{i|x|} \SS(i)$ are both dualizable in $\GrZ$ (see \cref{shift-dual}), hence, the free left $\SS[x]$-modules $\SS[x]$ and $\Sigma^{i|x|} \SS[x](i)$ are indeed left dualizable.

	For dualizability from the right, we recall from (the dual version of) \cite[Proposition 4.6.2.13]{HA} that it is equivalent to right dualizability of underlying graded spectrum.
	Applying \cref{quotient-induct} inductively, we see that as a graded spectrum, $\SS[x]/x^i$ is a finite direct sum of suspensions and weight shifts of the unit, each of them is dualizable (see \cref{shift-dual}), so it is right dualizable.
\end{proof}

In the remainder of this section we shall study modules over $\SS[x]$, and we record the following.

\begin{prop}\label{mod-sx-rigid}
	$\Mod_{\SS[x]}(\GrZ)$ is rigid.
\end{prop}

\begin{proof}
	Note that $\GrZ$ is rigid since all objects in $\ZZ$ are dualizable (that is, it is an abelian group), c.f.\ \cite[Example 4.39]{locrig}.
	The rigidity of $\Mod_{\SS[x]}(\GrZ)$ then follows from \cite[Proposition 1.9]{EfimovLimit} (see also the argument \cite[Example 4.38]{locrig}, which works in the $\bbE_2$-case as well).
\end{proof}

For our applications in the later sections, we shall be interested in the realization of the quotients in spectra via \cref{realization}, and furthermore, in the multi-variable case, and we introduce the following.

\begin{notn}\label{multi-quot}
	For a set of variables $\xx = (x_1,\dotsc,x_r)$ of even degrees $|x_1|,\dotsc,|x_r| \in 2\ZZ$, we denote the polynomial algebra on $\xx$ by
	\[
	    \mdef{\SS[\xx]} = \SS[x_1] \otimes \cdots \otimes \SS[x_r]
	    \qin \Alg_{\Et}(\Sp).
	\]
	For a tuple of non-negative integers $\ii = (i_1,\dotsc,i_r) \in \Ng^r$, we denote
	\[
		\mdef{\SS[\xx]/\xx^\ii} := \SS[x_1]/x_1^{i_1} \otimes \cdots \otimes \SS[x_r]/x_r^{i_r}
		\qin \Alg_{\Et}(\Sp),
	\]
	assembling into a diagram
	\[
		\SS[\xx] \too (\SS[\xx]/\xx^\ii)_{\ii \in \Ng^r}
		\qin \Alg_{\Et}(\Sp).
	\]
\end{notn}

\subsection{Pro-constancy of quotients}\label{subsec-pro-const}

In this subsection we prove the core result which powers the computation of the dualizable limit (that is, the strong Mittag-Leffler condition, see \cref{ML}) and the ordinary limit (see \cref{ordinary-limit}), as well as their comparison (see \cref{comp-nuc}).
This takes two forms, \cref{double-quotient-pro-sphere-Sx} and \cref{higher-quotient-base-Sx}, both of which follow from the more fundamental \cref{double-quotient-pro-sphere}.
To state it, we begin with the following.

\begin{const}\label{double-quotient-const}
	We construct a tower of augmented $\Eo$-algebras over $\SS \otimes_{\SS[x]} \SS$
	\[
		\SS \simeq \SS \otimes_{\SS[x]} \SS[x]
		\too (\SS \otimes_{\SS[x]} \SS[x]/x^i)_{i \geq 1}
		\qin \Algaug(\GrZ)_{/\SS \otimes_{\SS[x]} \SS}.
	\]
	Consider the tower of $\Et$-algebras from \cref{trunc-poly}
	\[
		\SS[x] \too (\SS[x]/x^i)
		\qin \Alg_{\Et}(\GrZ),
	\]
	and forget to $\Eo$-$\SS[x]$-algebras.
	The tower of $\Eo$-algebras is obtained by extension of scalars along the $\Et$-algebra map $\SS[x] \to \SS$, which takes $\Eo$-$\SS[x]$-algebras to $\Eo$-algebras.
	The augmentation comes from the observation that the weight truncation $w_{<1}$ is symmetric monoidal (\cref{weight-trunc-tower}), and takes all involved algebras to $\SS$.
	Finally, the bottom of the tower is $\SS \otimes_{\SS[x]} \SS$, so the tower is indeed over it.
\end{const}

\begin{remark}
	As we see in the construction, the augmentation and the map to $\SS \otimes_{\SS[x]} \SS$ are both encoded in the tower of $\Eo$-algebras.
	Nevertheless, both for the proof and the applications, it will be convenient to remember them.
\end{remark}

Before proving the pro-constancy result, we compute the terms in the tower, as graded spectra, and the space of maps between them.

\begin{lem}\label{double-quotient-identification}
	For every $i \geq 1$ we have
	\[
		\SS \otimes_{\SS[x]} \SS[x]/x^i \simeq \SS \oplus \Sigma^{i|x|+1} \SS(i)
		\qin \GrZ.
	\]
\end{lem}

\begin{proof}
	By \cref{quotient-cofiber}, we have an exact sequence
	\[
		\Sigma^{i|x|} \SS[x](i) \too[x^i] \SS[x] \too \SS[x]/x^i
		\qin \Mod_{\SS[x]}(\GrZ).
	\]
	Extending scalars along $\SS[x] \to \SS$, we get an exact sequence
	\[
		\Sigma^{i|x|} \SS(i) \too \SS \too \SS \otimes_{\SS[x]} \SS[x]/x^i
		\qin \GrZ.
	\]
	Observe that the first map is between graded spectra concentrated in different weights, and hence is null, and the result follows.
\end{proof}

\begin{lem}\label{contractible-homs}
	For $i > j \geq 1$ the following hom space is contractible
	\[
		\hom_{\Algaug(\GrZ)}(\SS \otimes_{\SS[x]} \SS[x]/x^i,\ \SS \otimes_{\SS[x]} \SS[x]/x^j)
		\quad\simeq\quad \pt.
	\]
	The same is true in the over category $\Algaug(\GrZ)_{/\SS \otimes_{\SS[x]} \SS}$.
\end{lem}

\begin{proof}
	By \cref{double-quotient-identification}, both objects lie in $\GrN$.
	Moreover, the target is supported on weights $0$ and $j$, which are smaller than $i$, so it is in fact in $\Gri \hookrightarrow \GrN$, while for the source we have
	\[
		w_{<i}(\SS \otimes_{\SS[x]} \SS[x]/x^i) \simeq \SS.
	\]
	Recall from \cref{weight-trunc-tower} that the weight truncation-inclusion adjunction is symmetric monoidal, so it induces an adjunction on augmented algebras.
	This, together with the fact that $\SS$ is the zero object in augmented algebras, proves the first part
	\begin{align*}
		\hom_{\Algaug(\GrN)}(\SS \otimesu_{\SS[x]} \SS[x]/x^i,\ \SS \otimesu_{\SS[x]} \SS[x]/x^j)
		&\simeq \hom_{\Algaug(\Gri)}(w_{<i}(\SS \otimesu_{\SS[x]} \SS[x]/x^i),\ \SS \otimesu_{\SS[x]} \SS[x]/x^j)\\
		&\simeq \hom_{\Algaug(\Gri)}(\SS,\ \SS \otimesu_{\SS[x]} \SS[x]/x^j)\\
		&\simeq \pt,
	\end{align*}
	Now, recall that the hom space in a general over category $\CC_{/B}$ is given by
	\[
		\hom_{/B}(X, Y) \ \simeq\ \fib(\hom(X, Y) \to \hom(X, B)).
	\]
	Thus the second part follows from the first part applied to $i>j$ and $i>1$.
\end{proof}

With this in place, we have the following key pro-constancy result.

\begin{prop}\label{double-quotient-pro-sphere}
	The map from \cref{double-quotient-const} induces an isomorphism of pro-objects
	\[
		c(\SS) \iso \prolim[i] (\SS \otimes_{\SS[x]} \SS[x]/x^i)
		\qin \Pro(\Algaug(\GrZ)_{/\SS \otimes_{\SS[x]} \SS}).
	\]
\end{prop}

\begin{proof}
	We begin by constructing a map in the other direction.
	Observe that the augmentation gives a map
	\[
		(\SS \otimes_{\SS[x]} \SS[x]/x^i)_i \too \SS
		\qin \Algaug(\GrZ).
	\]
	After restricting to $i > 1$ (which does not affect the pro-object), we claim that this is over $\SS \otimes_{\SS[x]} \SS$.
	Indeed, \cref{contractible-homs} applied to $i > 1$ shows that the space of maps from the $i$-th term to $\SS \otimes_{\SS[x]} \SS$ is contractible, so the maps to the bottom of the tower agree with the composite through the augmentation.
	Considering these as pro-objects, we get the required map in the other direction.

	We now show that both composites are homotopic to the identity, which will show that the map in question is indeed an isomorphism.
	In fact, we show that the endomorphism spaces of both the source and the target are contractible, and in particular both composites are (uniquely) homotopic to the identity.

	For the source, note that $\SS$ is the zero object in $\Algaug(\GrZ)$, hence its endomorphism space is contractible.
	This implies that its endomorphism space in $\Algaug(\GrZ)_{/\SS \otimes_{\SS[x]} \SS}$ is contractible as well, and hence so is the endomorphism space of $c(\SS)$ in the pro-category.

	For the target, by the formula for the hom space in a pro-category, the endomorphism space is
	\[
		\lim_{j}\ \colim_{i}\ \hom_{\Algaug(\GrZ)_{/\SS \otimes_{\SS[x]} \SS}}(\SS \otimes_{\SS[x]} \SS[x]/x^i,\ \SS \otimes_{\SS[x]} \SS[x]/x^j).
	\]
	By cofinality, we can restrict the colimit to be taken over $i > j$.
	\cref{contractible-homs} shows that all hom spaces involved are contractible, and hence the limit of the colimit is contractible as well, since $|\Ng| \simeq \pt$.
\end{proof}

We now draw two variants, which will be useful for our applications.
The first, looking at the same diagram, but taking place in a less structured category.

\begin{cor}\label{double-quotient-pro-sphere-Sx}
	The following map is an isomorphism
	\[
		c(\SS) \iso \prolim[i] (\SS \otimes_{\SS[x]} \SS[x]/x^i)
		\qin \Pro(\Alg(\Mod_{\SS[x]}(\GrZ))).
	\]
\end{cor}

\begin{proof}
	Notice that the $\SS[x]$-$\Eo$-algebra structure on this entire diagram is obtained by restriction of scalars along the $\Et$-map $\SS[x] \to \SS$, so the result follows from \cref{double-quotient-pro-sphere}.
\end{proof}

Second, we have the following base-changed variant.

\begin{cor}\label{quotient-base-Sx}
	The following map is an isomorphism
	\[
		c(\SS \otimes_{\SS[x]} \SS) \iso \prolim[i] (\SS \otimes_{\SS[x]/x^i} \SS)
		\qin \Pro(\BMod{\SS}{\SS}(\Mod_{\SS[x]}(\GrZ))).
	\]
\end{cor}

\begin{proof}
	Applying the functor
    \[
        \SS \otimes_{(-)} (\SS \otimes_{\SS[x]} \SS)\colon \Algaug(\GrZ)_{/\SS \otimes_{\SS[x]} \SS} \too \BMod{\SS}{\SS}(\Mod_{\SS[x]}(\GrZ)).
    \]
	to the isomorphism from \cref{double-quotient-pro-sphere}, we get an isomorphism
	\[
		c(\SS \otimesu_{\SS} (\SS \otimes_{\SS[x]} \SS))
		\too \prolim[i] (\SS \otimesu_{\SS \otimes_{\SS[x]} \SS[x]/x^i} (\SS \otimes_{\SS[x]} \SS))
		\qin \Pro(\BMod{\SS}{\SS}(\Mod_{\SS[x]}(\GrZ))).
	\]
	Noting that for an $\SS[x]$-algebra $A$ with an augmentation $A \to \SS$, we have a natural isomorphism
    \[
        \SS \otimes_{A} \SS
        \simeq \SS \otimesu_{\SS \otimes_{\SS[x]} A} (\SS \otimes_{\SS[x]} A) \otimes_{A} \SS
        \simeq \SS \otimesu_{\SS \otimes_{\SS[x]} A} (\SS \otimes_{\SS[x]} \SS),
    \]
	and applying this to $A = \SS[x]$ and $A = \SS[x]/x^i$, we get the desired isomorphism.
\end{proof}

In fact, this has the following strengthened version.

\begin{cor}\label{higher-quotient-base-Sx}
    For fixed $j$ and $k$, the following map is an isomorphism
	\[
		c(\SS[x]/x^j \otimesu_{\SS[x]} \SS[x]/x^k) \iso \prolim[i \geq j,k] (\SS[x]/x^j \otimesu_{\SS[x]/x^i} \SS[x]/x^k)
		\qin \Pro(\BMod{\SS[x]/x^j}{\SS[x]/x^k}(\Mod_{\SS[x]}(\GrZ))).
	\]
\end{cor}

\begin{proof}
	This follows by filtering $\SS[x]/x^j$ and $\SS[x]/x^k$ along their weight truncations.
	More precisely, consider the filtration
	\[
		\SS[x]/x^j \too \SS[x]/x^{j-1} \too \cdots \too \SS[x]/x^1 \too 0
		\qin \Mod_{\SS[x]/x^j}(\Mod_{\SS[x]}(\GrZ)).
	\]
	and similarly for $\SS[x]/x^k$.
	This yields a double filtration on the map in question.
	Since this (double) filtration has finite length, it suffices to show that the map is an isomorphism on the (double) associated graded, and then conclude by induction starting from the base case of $j=k=0$ where everything vanishes.

	Observe that the terms in the associated graded of the weight filtration on $\SS[x]/x^j$ are $\SS$, up to suspending and shifting weight (which commute with all relevant operations since they are given by tensoring with tensor invertible objects), and its $\SS[x]/x^j$-module structure is restricted along the map $\SS[x]/x^j \to \SS$.
	The same holds for $\SS[x]/x^k$.
	Thus, the conclusion follows from \cref{quotient-base-Sx} by restricting the module structure along the maps $\SS[x]/x^j \to \SS$ and $\SS[x]/x^k \to \SS$.
\end{proof}

This immediately implies the Mittag-Leffler condition for modules over the quotient tower.

\begin{thm}\label{ML}
	The tower
	\[
		\Mod_{\SS[x]}(\GrZ) \too (\Mod_{\SS[x]/x^\bullet}(\GrZ))
	\]
	satisfies the tower Mittag-Leffler condition in $\PrL_{\Mod_{\SS[x]}(\GrZ)}$, and the strong Mittag-Leffler condition over $\Mod_{\SS[x]}(\GrZ)$.
\end{thm}

\begin{proof}
	We first note that the tower
	\[
		\SS[x] \too (\SS[x]/x^\bullet)
	\]
	satisfies the tower (hence, by \cref{bML-cML}, the pretower) Mittag-Leffler condition in $\Mor(\Mod_{\SS[x]}(\GrZ))$.
	Indeed, the two conditions of \cref{cat-ML-morita} are satisfied by \cref{higher-quotient-base-Sx} and \cref{quotient-dualizable}.
	The result follows by \cref{ML-ring-to-mod}, noting that $\Mod_{\SS[x]}(\GrZ)$ is rigid by \cref{mod-sx-rigid}.
\end{proof}

	\section{Quotient towers and their module categories}\label{sec-quot}

In this subsection we study quotients of $\Et$-algebras and their module categories, building on the universal case from \cref{sec-gr}.
To start, we fix the setup, recalling that in \cref{multi-quot} we have introduced the multi-variable polynomial algebra $\SS[\xx]$.
We note that the different subsections will require increasing additional assumptions which we declare at their beginning (see \assref{dual-assum}, \assref{comp-assum} and \assref{assum-rigid-cg}).

\begin{namedassum}{Setup}\label{setup}
    We fix a presentably monoidal stable category $\EE \in \Alg(\PrLst)$ and an $\Et$-map
    \[
        \SS[\xx] := \SS[x_1] \otimes \cdots \otimes \SS[x_r] \too \End(\AA)
        \qin \Alg_{\Et}(\Sp)
    \]
    to the underlying $\Et$-ring spectrum of its unit, with $|x_1|,\dotsc,|x_r| \in 2\ZZ$ of even degrees.
\end{namedassum}

We now recall the main example we have in mind.
As explained in the introduction, this setup arises naturally for \emph{even} ring spectra, by the following remarkable result from \cite[Corollary 4.8]{RSS}, previously appearing for one variable in \cite[Proposition 3.13]{ABM}, with prior closely related results in \cite[\S5.4]{rotation}, \cite[Proposition 4.2.1]{HW} and \cite{E2quot}.

\begin{prop}[{\cite[Corollary 4.8]{RSS}}]\label{R-even}
	Let $R \in \Alg_{\Et}(\Sp)$ be an \emph{even} $\Et$-ring spectrum, and let $x_1,\dotsc,x_r \in \pi_*(R)$ be homogeneous elements of even non-negative degrees $|x_1|,\dotsc,|x_r| \in 2\NN$.
    Then there exists an $\Et$-map sending the polynomial generators to the given elements
    \[
        \SS[\xx] \too R
        \qin \Alg_{\Et}(\Sp).
    \]
\end{prop}

\begin{remark}
	In fact, the statement holds for any even degrees $|x_1|,\dotsc,|x_r| \in 2\ZZ$, by shearing the cellular structure from the case of degree zero, as in the one-variable proof of \cite[Proposition 3.13]{ABM}.

	More specifically, consider the $\Et$-map $\free_{\Et}(\SS\{y_1\}\oplus\cdots\oplus\SS\{y_r\}) \to \SS[y_1,\dotsc,y_r]$ in $\Gr^{\otimes_r}$, where $y_s$ has degree $0$ and multi-weight $e_s$.
	The proof of \cite[Proposition 4.6]{RSS} shows that the target is obtained from the source by attaching even $\Et$-cells.
	Applying the $\Et$-shearing equivalence $\mrm{sh}^{|x_1|,\dotsc,|x_r|}\colon \Gr^{\otimes r} \iso \Gr^{\otimes r}$ to this presentation, we get that $\SS[\xx]$ is obtained from $\free_{\Et}(\Sigma^{|x_1|}\SS\{x_1\}\oplus\cdots\oplus\Sigma^{|x_r|}\SS\{x_r\})$ by attaching $\Et$-cells in (possibly negative) even degrees, and we conclude as in \cite[Corollary 4.8]{RSS}.
\end{remark}

We thus have the following example.
The reader may wish to mentally replace $\EE$ and $\AA$ by $\Mod_R$ and $R$ in the rest of the paper.

\begin{example}\label{ex-of-setup}
    Let $R$ be an even $\Et$-ring spectrum, and let $x_1,\dotsc,x_r \in \pi_*(R)$ be elements of even non-negative degrees $|x_1|,\dotsc,|x_r| \in 2\NN$.
    The presentably monoidal category $\EE = \Mod_R$, together with any $\Et$-map $\SS[\xx] \to R$ from \cref{R-even}, gives an example for \assref{setup}.

    Note that $\Mod_R$ is rigid and compactly generated by $R$, and in particular dualizable, so satisfies \assref{dual-assum}.
    If $R$ moreover has an $\Eth$-ring spectrum structure, then $\Mod_R$ satisfies \assref{comp-assum} and \assref{assum-rigid-cg} as well.
\end{example}

\begin{remark}
    While this is essentially the only example the author had in mind, the author nevertheless found it useful to work in the generality of \assref{setup} for two reasons.
    First, it clarifies which results depend on which properties of the example.
    Second, in earlier stages of this project, the author was at times confused why certain constructions had an $R$-module structure, or, more often, why multiple $R$-module structures coincided, which are manifest when working internally to $\EE$.
\end{remark}

\subsection{Quotients of \texorpdfstring{$\Et$}{E2}-algebras}\label{subsec-e2-quot}

We now turn our attention to constructing quotients of general $\Et$-algebras in our \assref{setup}.
Extension of scalars along the $\Et$-map $\SS[\xx] \to \AA$ takes $\Eo$-$\SS[\xx]$-algebras to $\Eo$-algebras in $\EE$, namely, gives a functor
\[
	\AA \otimes_{\SS[\xx]} (-)\colon \Alg(\Mod_{\SS[\xx]}(\Sp)) \too \Alg(\EE),
\]
and we give the following definition.

\begin{defn}\label{quotient-def}
	For a tuple $\ii \in \Ng^r$, we define the quotient of $\AA$ by $\xx^\ii$ to be the image of $\SS[\xx]/\xx^\ii$ from \cref{multi-quot} under the functor above
	\[
		\mdef{\AA/\xx^\ii} := \AA \otimes_{\SS[\xx]} \SS[\xx]/\xx^\ii
		\qin \Alg(\EE).
	\]
	As in \cref{multi-quot}, these assemble into a diagram
	\[
		\AA \too (\AA/\xx^\ii)_{\ii \in \Ng^r}
		\qin \Alg(\EE).
	\]
	We also use the following notation
	\[
		\mdef{\Ax} := \AA/\xx
		\qin \Alg(\EE).
	\]
\end{defn}

One difference from the graded case is that the variables no longer carry weight.
This manifests in the following result.

\begin{prop}\label{thick}
	For any $\ii \in \Ng^r$, the quotient $\AA/\xx^\ii$ is in the thick subcategory of $\BMod{\AA/\xx^\ii}{\AA/\xx^\ii}$ generated by $\Ax$.
\end{prop}

\begin{proof}
	Applying \cref{quotient-induct} inductively, we see that $\SS[x]/x^i$ is built from suspensions and weight shifts of $\SS[x]/x \simeq \SS$ in $\BMod{\SS[x]/x^i}{\SS[x]/x^i}(\GrZ)$.
	After realizing to spectra, the weight shifts become trivial, so $\SS[x]/x^i$ is in the thick subcategory generated by $\SS[x]/x \simeq \SS$ in $\Mod{\SS[x]/x^i}{\SS[x]/x^i}(\Sp)$.
	The general case follows by tensoring to the multi-variable case and base-changing to $\AA$.
\end{proof}

We then have the following useful consequence.

\begin{cor}\label{conservative}
	The extension of scalars functor
	\[
		\Ax \otimes_{\AA/\xx^\ii} (-) \colon \Mod_{\AA/\xx^\ii} \too \Mod_{\Ax}
	\]
	is conservative.
\end{cor}

\begin{proof}
	Fix a left $\AA/\xx^\ii$-module $M$.
	The collection of right $\AA/\xx^\ii$-modules $N$ such that $N \otimes_{\AA/\xx^\ii} M \simeq 0$ is a thick subcategory of $\BMod{\AA/\xx^\ii}{\AA/\xx^\ii}$.
	Thus, if $\Ax \otimes_{\AA/\xx^\ii} M \simeq 0$, then $M \simeq \AA/\xx^\ii \otimes_{\AA/\xx^\ii} M \simeq 0$ by \cref{thick}.
\end{proof}

We record the following elementary computation.

\begin{prop}\label{pi0}
	Let $R \in \Alg_{\Et}(\Spcn)$ be equipped with an $\Et$-map $\SS[\xx] \to R$ with non-negative even degrees $|x_1|,\dotsc,|x_r| \in 2\NN$.
	Then $R/\xx^\ii$ is connective, and its $0$-th homotopy group is the quotient by the degree-zero elements
	\[
		\pi_0(R/\xx^\ii) \:\simeq\: \pi_0(R)\,/\,(x_s^{i_s}\,\mid\,|x_s|=0).
	\]
\end{prop}

\begin{proof}
	Recall that
	\[
		R/\xx^\ii \simeq R/x_1^{i_1} \otimes_R \cdots \otimes_R R/x_r^{i_r}.
	\]
	We argue by induction on the variables.
	Denote by $R_s$ the tensor product of the first $s$ terms, and note that the claim holds for the base-case $R_0 = R$.
	\cref{quotient-cofiber} shows that we have an exact sequence
	\[
		\Sigma^{i_s |x_s|} R \too[x_s^{i_s}] R \too R/x_s^{i_s},
	\]
	so tensoring with $R_{s-1}$ over $R$ gives an exact sequence
	\[
		\Sigma^{i_s |x_s|} R_{s-1} \too[x_s^{i_s}] R_{s-1} \too R_s.
	\]
	Inspecting the long exact sequence of homotopy groups, we get that $R_s$ is connective and
	\[
		\pi_0(R_s) \simeq
		\begin{cases}
			\pi_0(R_{s-1}) & \quad|x_s|>0, \\
			\pi_0(R_{s-1})/(x_s^{i_s}) & \quad|x_s|=0,
		\end{cases}
	\]
	and the conclusion follows.
\end{proof}

We now deduce the Mittag-Leffler condition.
Recall that this is a condition on (pre)towers, rather than on general diagrams, so we will restrict to the diagonal $\ci = (i,\dotsc,i)$, which is cofinal in the poset of indices $\Ng^r$.

\begin{prop}\label{cat-ML-quotient}
	$\EE \to (\Mod_{\AA/\xx^{\ci}})$ is a Mittag-Leffler tower in $\PrLE$ and in $\PrLst$.
\end{prop}

\begin{proof}
	By \cref{ML} we know that
	\[
		\Mod_{\SS[x_s]}(\GrZ) \too (\Mod_{\SS[x_s]/x_s^\bullet}(\GrZ))
		\qin \PrL_{\Mod_{\SS[x_s]}(\GrZ)}
	\]
	is a Mittag-Leffler tower for every $1 \leq s \leq r$.
	By \cref{ML-2}, we get that the image under the $2$-functor
	\[
		\prod_{s=1}^r \PrL_{\Mod_{\SS[x_s]}(\GrZ)}
		\too \prod_{s=1}^r \PrL_{\Mod_{\SS[x_s]}(\Sp)}
		\too \PrL_{\Mod_{\SS[\xx]}(\Sp)}
		\too \PrL_{\EE},
	\]
	given by extension of scalars and tensor product, is indeed a Mittag-Leffler tower in $\PrLE$.
	Similarly, further applying the $2$-functor $\PrLE \to \PrLst$ gives the result in $\PrLst$.
\end{proof}

\subsection{Completion and the ordinary limit}\label{subsec-comp}

In this subsection we study the ordinary limit, showing that it is equivalent to complete modules.
We refer the reader to \cite{BHV} for an extended study of completion and local duality in higher algebra, noting that \emph{loc.\ cit.} is written in the symmetric monoidal context, though many of its results apply under weaker monoidality assumptions.
We begin by recalling \cref{u-acyclic-local} for the present tower.

\begin{defn}\label{Ax-equiv-acyc-loc-defs}
	We say that $N \in \EE$ is \tdef{$\Ax$-acyclic} if $\Ax \otimesA N = 0$.
	We say that $M \in \EE$ is \tdef{$\xx$-complete} if it is \tdef{$\Ax$-local}, that is, if
	\[
		\hom(N, M) = 0
		\qquad \text{for all $\Ax$-acyclic $N \in \EE$},
	\]
	and we denote their full subcategory by $\mdef{\EEc} \subset \EE$.
\end{defn}

We immediately deduce the computation of the ordinary limit using the Mittag-Leffler condition.
We refer the reader to \cref{related-work} for a discussion of related previous results.

\begin{prop}\label{ordinary-limit}
	The ordinary limit is the $\xx$-complete objects
	\[
		\EEc \iso \lim_\ii \Mod_{\AA/\xx^\ii}
		\qin \PrLE.
	\]
	In particular, $\EEc$ is a reflective localization of $\EE$, and the localization functor is given by
	\[
		\mdef{\xc{(-)}}\colon \EE \too \EEc,
		\qquad
		M \mapstoo \lim_\ii (\AA/\xx^\ii \otimesA M).
	\]
\end{prop}

\begin{proof}
	By \cref{bML-ordinary-limit} applied to \cref{cat-ML-quotient}, the functor
	\[
		\EE \too \lim_\ii \Mod_{\AA/\xx^\ii}
	\]
	is a reflective localization.
	By \cref{conservative}, the maps in the diagram are conservative, so \cref{complete-local-conservative} shows that the essential image of the right adjoint is the $\xx$-complete objects, concluding the proof.
\end{proof}

We record the following consequence for a monoidal structure on the complete objects, in the case $\EE$ is $\Et$-monoidal.
We remark that if $\EE$ is only assumed to be $\Eo$-monoidal, the argument below for closure from the left does not work, and we do not know if $\EEc$ acquires a monoidal structure in that case.

\begin{prop}\label{Ec-E2-monoidal}
	If $\EE$ is $\Et$-monoidal, then $\EEc$ has a presentably $\Et$-monoidal structure, and $\xx$-completion is an $\Et$-monoidal functor.
\end{prop}

\begin{proof}
	We start by proving that $\xx$-completion is compatible with the $\Et$-monoidal structure in the sense of \cite[Definition 2.2.1.6]{HA}, which implies that $\EEc$ is $\Et$-monoidal and $\xx$-completion is $\Et$-monoidal by \cite[Proposition 2.2.1.9]{HA}.
	Thus, we need to show that for any $n$-tuple of $\Ax$-equivalences
	\[
		f_1\colon M_1 \too M_1',\ \dotsc,\ f_n\colon M_n \too M_n'
	\]
	and $n$-ary operation $\mu \in \Et(n)$, the tensor product $\bigotimes_\mu(f_1,\dotsc,f_n)$ is a $\Ax$-equivalence as well.
	Since $\Et(n)$ is connected, it suffices to prove this for the ordered tensor product $f_1 \otimesA \cdots \otimesA f_n$.
	Consider the maps
	\[
		\widetilde{f}_i: M_1' \otimesA \cdots \otimesA M_{i-1}' \otimesA M_i \otimesA M_{i+1} \otimesA \cdots \otimesA M_n
		\too M_1' \otimesA \cdots \otimesA M_{i-1}' \otimesA M_i' \otimesA M_{i+1} \otimesA \cdots \otimesA M_n
	\]
	defined as $\widetilde{f}_i = \Id \otimes f_i \otimes \Id$.
	By the distributivity of the tensor product over composition we have
	\[
		f_1 \otimesA \cdots \otimesA f_n
		\simeq \widetilde{f}_n \circ \cdots \circ \widetilde{f}_1,
	\]
	so it suffices to show that each $\widetilde{f}_i$ is a $\Ax$-equivalence.
	In other words, we need to show that $\Ax$-equivalences are closed under tensoring with arbitrary objects from both sides.
	Closure from the right side follows from the fact that a map $g$ is a $\Ax$-equivalence if and only if $\Ax \otimesA g$ is an equivalence.
	Closure from the left side follows in the same way using the braiding of $\EE$.
	This concludes the proof of compatibility of $\xx$-completion with the $\Et$-monoidal structure, making it $\Et$-monoidal.
	Finally, the fact that the localized $\Et$-monoidal structure on $\EEc$ is presentable, i.e.\ commutes with colimits in each coordinate, follows from the fact that the localization is a left adjoint.
\end{proof}

\begin{defn}
	Let $\BB \in \Alg(\EE)$.
	We define the category of \tdef{$\xx$-complete $\BB$-modules} to be those $\BB$-modules whose underlying object is $\xx$-complete, and we denote their full subcategory by
	\[
		\mdef{\cMod_{\BB}} \quad\subseteq\quad \Mod_{\BB}.
	\]
\end{defn}

\begin{remark}
	If $\EE$ is $\Et$-monoidal, then by \cref{Ec-E2-monoidal} we also have the following identifications
	\[
		\cMod_{\BB} \ \simeq\ \Mod_{\BB} \otimes_\EE \EEc \ \simeq\ \Mod_{\BB}(\EEc) \ \simeq\ \Mod_{\xc{\BB}}(\EEc).
	\]
\end{remark}

\begin{prop}\label{auto-complete}
	The underlying object of any left $\AA/\xx^\ii$-module is $\xx$-complete.
	If $\EE$ is $\Et$-monoidal, then the same is true for right $\AA/\xx^\ii$-modules.
\end{prop}

\begin{proof}
	Let $M$ be a left $\AA/\xx^\ii$-module, and let $N \in \EE$ be $\Ax$-acyclic.
	By \cref{thick}, $\AA/\xx^\ii$ is in the thick subcategory of $\Ax$, hence $N$ is also $\AA/\xx^\ii$-acyclic, and we conclude
	\[
		\hom(N, M) \simeq \hom_{\AA/\xx^\ii}(\AA/\xx^\ii \otimesA N, M) = 0.
	\]
	The case of right modules follows from the left module case via the braiding of $\EE$.
\end{proof}

Observe that the completion of the unit inherits an algebra structure, from the description of the completion functor as a limit
\[
	\xc{\AA} \simeq \lim_\ii \AA/\xx^\ii
	\qin \Alg(\EE).
\]
We will attend to this algebra in the next proposition, but we first make the following comment.

\begin{warn}\label{non-complete-modules}
	We remind the reader that, as is often the case for completion, there is a distinction between $\xx$-complete objects and modules over the $\xx$-completion of the unit which are not necessarily $\xx$-complete.
	Namely, there is a (usually proper) inclusion
	\[
		\EEc = \cMod_{\xc{\AA}} \quad\subset\quad \Mod_{\xc{\AA}}.
	\]
\end{warn}

We now consider the restriction of scalars along $\xc{\AA} \to \AA/\xx^\ii$ for both variants from \cref{non-complete-modules}.

\begin{prop}\label{quotient-compact}
	The functors of restriction of modules and $\xx$-complete modules along $\xc{\AA} \to \AA/\xx^\ii$
	\[
		\Mod_{\AA/\xx^\ii} \too \Mod_{\xc{\AA}},
		\qquad
		\cMod_{\AA/\xx^\ii} \too \EEc
	\]
	preserve compact objects.
\end{prop}

\begin{proof}
	It suffices to prove that $\AA/\xx^\ii$ is a left dualizable $\xc{\AA}$-module, since then both restriction of scalars are strongly continuous, hence compact object preserving.
	
	We first note that $\AA/\xx^\ii$ is dualizable in $\EE$ by \cref{quotient-dualizable}.
	Since it is dualizable, tensoring with it commutes with limits.
	Using the formula for $\xx$-completion from \cref{ordinary-limit}, together with the fact that $\AA/\xx^\ii$ is $\xx$-complete by \cref{auto-complete}, we get
	\[
		\xc{\AA} \otimesA \AA/\xx^\ii
		\simeq (\lim_\jj \AA/\xx^\jj) \otimesA \AA/\xx^\ii
		\simeq \lim_\jj (\AA/\xx^\jj \otimesA \AA/\xx^\ii)
		\simeq \xc{\AA/\xx^\ii}
		\simeq \AA/\xx^\ii.
	\]
	Now, extension of scalars along $\AA \to \xc{\AA}$ sends dualizable modules to dualizable modules, so using the fact that $\AA/\xx^\ii$ is dualizable in $\EE$ again, we conclude that it is also dualizable as a $\xc{\AA}$-module.
\end{proof}

As another consequence, we have the following compact generation statement.

\begin{cor}\label{compact-gen}
	If $\EE$ is compactly generated then so is $\EEc$.
\end{cor}

\begin{proof}
	First, we note that the projection map
	\[
		\EEc \too \Mod_{\Ax}
	\]
	is conservative and has a left adjoint.
	The fact that it is conservative follows from the identification of $\xx$-complete objects with $\Ax$-local objects.
	To see that it has a left adjoint note that it is the composite of the inclusion $\EEc \hookrightarrow \EE$ and the projection $\EE \to \Mod_{\Ax}$.
	The former has a left adjoint since it is a reflective localization (\cref{ordinary-limit}).
	For the latter, recall that the tower Mittag-Leffler condition is satisfied by \cref{cat-ML-quotient}, which in particular says that $\EE \to \Mod_{\Ax}$ has a left adjoint.

	Now, $\Mod_{\Ax}$ is compactly generated by the free modules on the compact generators of $\EE$.
	The left adjoint to the projection $\EEc \to \Mod_{\Ax}$ preserves compact objects since it is strongly continuous, and since the projection is conservative, the image of compact objects generates $\EEc$, concluding the proof.
\end{proof}

\subsection{Nuclear modules and their K-theory}\label{subsec-nuc}

In this section we study the dualizable limit and the strong Mittag-Leffler condition for modules over the quotients.
In particular, we need the following additional assumption on our \assref{setup}.

\begin{namedassum}{Dualizability Assumption}\label{dual-assum}
	The category $\EE$ is dualizable in $\PrLst$.
\end{namedassum}

\begin{prop}\label{dual-dual}
	$\Mod_{\AA/\xx^{\ii}} \in \PrLst$ is dualizable for all $\ii$.
\end{prop}

\begin{proof}
	$\EE$ is dualizable by the \assref{dual-assum}, so this follows from \cite[Corollary 1.55]{DualCat} applied to $\mcl{V} = \Sp$ and the monad $T = \AA/\xx^{\ii} \otimesA -$.
\end{proof}

\begin{defn}\label{nuc-def}
	We define the category of \tdef{$\xx$-nuclear objects of $\EE$} to be the dualizable limit
	\[
		\mdef{\NucE} := \limdual[\ii] \Mod_{\AA/\xx^\ii}
		\qin \PrLstdbl.
	\]
\end{defn}

\begin{prop}\label{ML-quotient-Sp}
	The pretower $(\Mod_{\AA/\xx^{\cbullet}})$ satisfies the strong Mittag-Leffler condition over $\Sp$.
	
	More generally, for any dualizable $\DD \in \PrLstdbl$, the pretower $(\Mod_{\AA/\xx^{\cbullet}} \otimes \DD)$ satisfies the strong Mittag-Leffler condition over $\Sp$ as well.
\end{prop}

\begin{proof}
	For the first part, the pretower is dualizable in $\PrLst$ by \cref{dual-dual}, and so the conclusion follows from \cref{cat-ML-quotient} and \cref{ML-equiv} (where we take the base to be $\Sp$).
	The second part follows from the first using \cite[Proposition 5.3(iii)]{EfimovLimit}.
\end{proof}

Although we will not use it in this paper, we note that if $\EE$ is rigid, so that the strong Mittag-Leffler condition over it is defined, we immediately get the following.

\begin{prop}\label{ML-quotient-rigid}
	If $\EE$ is rigid in $\PrLst$ then $(\Mod_{\AA/\xx^{\cbullet}})$ satisfies the strong Mittag-Leffler condition over $\EE$.
\end{prop}

\begin{proof}
	Note that each $\Mod_{\AA/\xx^\ii}$ is dualizable in $\PrLE$ by \cite[Remark 4.8.4.8]{HA}, so this follows from \cref{cat-ML-quotient} and \cref{ML-equiv}.
\end{proof}

Using Efimov's work on inverse limits \cite{EfimovLimit} we deduce the following main result.
We remind the reader that we are working in our \assref{setup} and under the \assref{dual-assum}.

\begin{thm}[{\cref{k-nuc-intro}}]\label{k-nuc}
	The assembly map induces an isomorphism
	\[
		\Kcont(\NucE) \iso \lim_\ii \Kcont(\Mod_{\AA/\xx^\ii}).
	\]
	More generally, for any dualizable $\DD \in \PrLstdbl$, the assembly map induces an isomorphism
	\[
		\Kcont(\NucE \otimes \DD) \iso \lim_\ii \Kcont(\Mod_{\AA/\xx^\ii} \otimes \DD).
	\]
\end{thm}

\begin{proof}
	By cofinality, it suffices to restrict to the diagonal $\ci = (i,\dotsc,i)$, and the result follows from \cite[Corollary 6.2]{EfimovLimit} applied to \cref{ML-quotient-Sp}.
\end{proof}

Restricting to the case of connective ring spectra, we get the following consequence by applying the Dundas--Goodwillie--McCarthy theorem.

\begin{cor}\label{nuc-dgm}
	Let $R \in \Alg_{\Et}(\Spcn)$ be equipped with an $\Et$-map $\SS[\xx] \to R$ with non-negative even degrees $|x_1|,\dotsc,|x_r| \in 2\NN$.
	Then, there is a pullback square
	% https://q.uiver.app/#q=WzAsNCxbMCwwLCJcXEtjb250KFxcTnVjUikiXSxbMCwxLCJcXEtLKFxcQXgpIl0sWzEsMCwiXFxsaW1fXFxpaSBcXFRDKFIvXFx4eF5cXGlpKSJdLFsxLDEsIlxcVEMoXFxBeCkiXSxbMSwzXSxbMiwzXSxbMCwxXSxbMCwyXSxbMCwzLCIiLDEseyJzdHlsZSI6eyJuYW1lIjoiY29ybmVyIn19XV0=
	\[\begin{tikzcd}
		{\Kcont(\NucR)} & {\lim_\ii \TC(R/\xx^\ii)} \\
		{\KK(\Ax)} & {\TC(\Ax)}
		\arrow[from=1-1, to=1-2]
		\arrow[from=1-1, to=2-1]
		\arrow["\lrcorner"{anchor=center, pos=0.125}, draw=none, from=1-1, to=2-2]
		\arrow[from=1-2, to=2-2]
		\arrow[from=2-1, to=2-2]
	\end{tikzcd}\]
\end{cor}

\begin{proof}
	For a fixed $\ii$, consider the map $R/\xx^\ii \to \Ax$.
	By \cref{pi0}, the induced map on $\pi_0$ is
	\[
		\pi_0(R/\xx^\ii)
		\:\simeq\: \pi_0(R)\,/\,(x_s^{i_s}\,\mid\,|x_s|=0)
		\too \pi_0(R)\,/\,(x_s\,\mid\,|x_s|=0)
		\:\simeq\: \pi_0(\Ax).
	\]
	We see that this map is a surjection on $\pi_0$ with nilpotent kernel, and so by \cite[Theorem 7.2.2.1]{DGM} the following is a pullback square
	% https://q.uiver.app/#q=WzAsNCxbMCwwLCJcXEtLKFIvXFx4eF5cXGlpKSJdLFswLDEsIlxcS0soXFxBeCkiXSxbMSwwLCJcXFRDKFIvXFx4eF5cXGlpKSJdLFsxLDEsIlxcVEMoXFxBeCkiXSxbMSwzXSxbMiwzXSxbMCwxXSxbMCwyXSxbMCwzLCIiLDIseyJzdHlsZSI6eyJuYW1lIjoiY29ybmVyIn19XV0=
	\[\begin{tikzcd}
		{\KK(R/\xx^\ii)} & {\TC(R/\xx^\ii)} \\
		{\KK(\Ax)} & {\TC(\Ax)}
		\arrow[from=1-1, to=1-2]
		\arrow[from=1-1, to=2-1]
		\arrow["\lrcorner"{anchor=center, pos=0.125}, draw=none, from=1-1, to=2-2]
		\arrow[from=1-2, to=2-2]
		\arrow[from=2-1, to=2-2]
	\end{tikzcd}\]
	The result now follows by passing to the limit over $\ii$ and applying \cref{k-nuc}.
\end{proof}

Using the flexibility of \cref{k-nuc} to tensoring with $\DD$, we give the following two examples.

\begin{example}[Functor categories]\label{fun-nuc}
	For a (small) category $\CC$, the functor category $\DD = \Fun(\CC, \Sp)$ is a compactly generated, hence dualizable, presentable stable category, and $\Fun(\CC, \EE) \simeq \DD \otimes \EE$.
	Therefore, by \cref{k-nuc} we have an isomorphism
	\[
		\Kcont(\Fun(\CC, \NucE)) \iso \lim_\ii \Kcont(\Fun(\CC, \Mod_{\AA/\xx^\ii})).
	\]
\end{example}

\begin{example}[Sheaves]\label{sh-nuc}
	For a locally compact Hausdorff space $Y$, the category of sheaves $\DD = \Shv(Y, \Sp)$ is a dualizable presentable stable category, and $\Shv(Y, \EE) \simeq \Shv(Y, \Sp) \otimes \EE$ (see for example \cite[Proposition 2.54]{DualCat}, \cite[Example 1.28]{EfimovK} or \cite[Theorem 5.15]{sixtop}).
	Therefore, by \cref{k-nuc} we have an isomorphism
	\[
		\Kcont(\Shv(Y, \NucE)) \iso \lim_\ii \Kcont(\Shv(Y, \Mod_{\AA/\xx^\ii})),
	\]
	which \cite[Theorem 0.2]{EfimovK} identifies with
	\[
		\Gammac(Y, \Kcont(\NucE)) \iso \lim_\ii \Gammac(Y, \Kcont(\Mod_{\AA/\xx^\ii})).
	\]
\end{example}

\subsection{Comparison of complete and nuclear modules}\label{subsec-comp-nuc}

Our next goal is to compare the ordinary limit and the dualizable limit studied in the previous two subsections.
To make sense of that, we must work under the \assref{dual-assum} on our \assref{setup}.
Moreover, our proof also requires $\EE$ to be $\Et$-monoidal, as was highlighted in \cref{outline}.

\begin{namedassum}{Comparison Assumption}\label{comp-assum}
	The category $\EE$ is $\Et$-monoidal and dualizable in $\PrLst$.
\end{namedassum}

Consider the inclusion
\[
	\PrLstdbl \quad\subset\quad \PrLst.
\]
By \cref{dual-dual}, our diagram of interest $(\Mod_{\AA/\xx^\ii})_\ii$ is in $\PrLstdbl$.
Recall that the dualizable limit is $\NucE$ by \cref{nuc-def} and the ordinary limit is $\EEc$ by \cref{ordinary-limit}.

\begin{defn}\label{q-def}
	We denote the corresponding assembly maps by
	\[
		\mdef{q}\colon \NucE \too \EEc
		\qin \PrLst.
	\]
\end{defn}

\begin{prop}\label{iota}
	$\EEc \in \PrLst$ is dualizable, and the functor $q$ admits a fully faithful left adjoint
	\[
		\mdef{\iota}\colon \EEc \too \NucE.
	\]
\end{prop}

\begin{remark}
	This does not use the $\Et$-monoidality.
\end{remark}

\begin{proof}
	The pretower $(\Mod_{\AA/\xx^\ii})_\ii$ satisfies the strong Mittag-Leffler condition over $\Sp$ by \cref{ML-quotient-Sp}.
	The dualizability is the content of \cite[Proposition 5.5(ii)]{EfimovLimit}, and the existence of the fully faithful left adjoint $\iota$ follows from \cite[Remark 5.15(i)]{EfimovLimit} (see also \cite[Proposition 5.5(i)]{EfimovLimit}).
\end{proof}

Our next goal is to show that the Verdier quotient of $\iota$, or, equivalently, the kernel of $q$, is $\Ax$-acyclic.
To that end we shall need the following comparison maps.
We recall that by \assref{comp-assum}, $\EE$ is $\Et$-monoidal, which in particular means that $\Alg(\EE)$ is monoidal, allowing us to define the following.

\begin{notn}
	Tensoring the diagram $(\AA/\xx^\ii)_\ii$ with the map $f\colon \AA \to \Ax$, we obtain maps
	\[
		(\mdef{f_\ii}) = (f \otimes \Id_{\AA/\xx^\ii})\colon (\AA/\xx^\ii) \too (\Ax \otimesA \AA/\xx^\ii)
		\qin \Alg(\EE)^{\Ng^r}.
	\]
\end{notn}

Taking modules, we get an induced map of diagrams
\[
	(f_{\ii!})\colon (\Mod_{\AA/\xx^\ii}) \too (\Mod_{\Ax \otimesA \AA/\xx^\ii})
	\qin (\PrLEdbl)^{\Ng^r}.
\]
We claim that this has a right adjoint, in the $2$-category of diagrams, and we can identify the associated monad.
We prove the following preliminary result.

\begin{prop}\label{extension-restriction-to-Ax}
	The $1$-morphism $(f_{\ii!}) \in (\PrLEdbl)^{\Ng^r}$ has a right adjoint $(f_\ii^*)$.
	Moreover, the associated monad is given by the right action of $\Ax$, namely
	\[
		(f_\ii^* f_{\ii!}) \simeq (-) \otimes \Ax.
	\]
\end{prop}

\begin{remark}
	To unpack the first part, we already know that each $f_{\ii!}$ has a right adjoint $f_\ii^*$ in $\PrLE$, and we are claiming that (a) $f_\ii^* \in \PrLEdbl$, that is, its further right adjoint $f_{\ii*}$ is colimit preserving and strongly $\EE$-linear, and (b) $(f_{\ii!})$ satisfy the Beck--Chevalley condition, that is $(f_\ii^*)$ assembles into a map of diagrams.
\end{remark}

\begin{proof}
	We start by noting that
	\[
		f_!\colon \EE \too \Mod_{\Ax}
		\qin \PrLEdbl
	\]
	is a left adjoint $1$-morphism, that is, its right adjoint $f^*$ is a morphism in $\PrLEdbl$.
	This follows from the fact that the corresponding $\AA$-$\Ax$-bimodule $\Ax$ is left dualizable by \cref{quotient-dualizable}, so by Morita theory (see \cite[Proposition 4.6.2.10]{HA}) the right adjoint $f^*$ is a morphism in $\PrLEdbl$ as well.
	Moreover, using the braiding from the \assref{comp-assum}, we get
	\[
		f^* f_! \simeq \Ax \otimesA (-) \simeq (-) \otimesA \Ax.
	\]
	This implies the result by tensoring.
	In more detail, note that for any monoidal $2$-category $\bbA$ and indexing category $I$, consider the $2$-category $\bbA^I$ endowed with the pointwise monoidal structure.
	The binary tensor product of $\bbA^I$ is a $2$-functor, and so is the diagonal $\bbA \to \bbA^I$, so together we get a $2$-functor
	\[
		\bbA \times \bbA^I \too \bbA^I \times \bbA^I \too \bbA^I,
	\]
	which in particular sends adjunctions to adjunctions.
	Applying this to $\bbA = \PrLEdbl$, $I = \Ng^r$, and the left adjoint $1$-morphisms $f_!$ and the identity of $(\Mod_{\AA/\xx^\ii})$, we get that $(f_{\ii!})$ is a left adjoint $1$-morphism in $(\PrLEdbl)^{\Ng^r}$, and, similarly, identify the monad.
\end{proof}

\begin{thm}\label{comp-nuc}
	The Verdier quotient of $\iota\colon \EEc \hookrightarrow \NucE$ is $\Ax$-acyclic.
	That is, for every $M \in \NucE$, the tensor product $M \otimesA \Ax$ belongs to the essential image of $\iota$.
\end{thm}

\begin{proof}
	Note that the Verdier quotient of $\iota$ is equivalent to $\ker(q) \subset \NucE$.
	Thus, equivalently, we show that if $M \in \NucE$ is in $\ker(q)$, then $M \otimesA \Ax = 0$.

	We shall consider the dualizable and ordinary limits, and the assembly map between them, of the map of diagrams
	\[
		(f_{\ii!})\colon (\Mod_{\AA/\xx^\ii}) \too (\Mod_{\Ax \otimesA \AA/\xx^\ii})
		\qin (\PrLstdbl)^{\Ng^r}.
	\]
	Observe that the source has dualizable limit $\NucE$ by \cref{nuc-def}, and ordinary limit $\EEc$ by \cref{ordinary-limit}.
	Consider now the target.
	\cref{double-quotient-pro-sphere-Sx} shows that the map of pro-objects
	\[
		c(\Ax) \too \prolim[i] (\Ax \otimesA \AA/\xx^\ii)
		\qin \Pro(\Alg(\EE))
	\]
	is an isomorphism, thus, applying
	\[
		\Mod_{(-)}(\EE)\colon \Alg(\EE) \too \PrLstdbl,
	\]
	we conclude that both the dualizable and the ordinary limit of the target of $(f_{\ii!})$ are equivalent to $\Mod_{\Ax}$.
	Combining these observations, we get a commutative square
	% https://q.uiver.app/#q=WzAsNCxbMCwwLCJcXE51Y1IiXSxbMSwwLCJcXGNNb2RfUiJdLFswLDEsIlxcTW9kX3tcXFJ4fSJdLFsxLDEsIlxcTW9kX3tcXFJ4fSJdLFsyLDMsIiIsMCx7ImxldmVsIjoyLCJzdHlsZSI6eyJoZWFkIjp7Im5hbWUiOiJub25lIn19fV0sWzAsMSwicSJdLFswLDIsImZfIV5cXE51YyIsMl0sWzEsMywiXFx3aWRlaGF0e2Z9XyEiXV0=
	\[\begin{tikzcd}[column sep=large]
		\NucE & {\EEc} \\
		{\Mod_{\Ax}} & {\Mod_{\Ax}}
		\arrow["q", from=1-1, to=1-2]
		\arrow["{f_!^\Nuc}"', from=1-1, to=2-1]
		\arrow["{\widehat{f}_!}", from=1-2, to=2-2]
		\arrow[equals, from=2-1, to=2-2]
	\end{tikzcd}\]
	Now, \cref{extension-restriction-to-Ax} shows that $(f_{\ii!})$ has a right adjoint $(f_\ii^*)$ in the $2$-category $(\PrLstdbl)^{\Ng^r}$, and the (dualizable) limit is a $2$-functor, so we get that $f_!^\Nuc$ has a right adjoint $f^{*\Nuc}$.
	Moreover, it shows that the associated monad is
	\[
		f^{*\Nuc} f_!^\Nuc \simeq (-) \otimesA \Ax.
	\]
	Combining this with the commutative square, we get that
	\[
		f^{*\Nuc} \widehat{f}_! q \simeq (-) \otimesA \Ax,
	\]
	hence $\ker(q) \subset \ker((-) \otimesA \Ax)$, concluding the proof.
\end{proof}

\subsection{The compactly generated limit}\label{subsec-cg}

In this subsection we study the compactly generated limit, so we additionally assume that $\EE$ is compactly generated, and our proofs also require rigidity, on top of our \assref{setup}.

\begin{namedassum}{Rigidity Assumption}\label{assum-rigid-cg}
	The category $\EE$ is $\Et$-monoidal, rigidly compactly generated.
	We denote by $\{C_\alpha\}$ a set of compact generators.
\end{namedassum}

Note that this says that for an object $M \in \EE$, the conditions of being compact, left dualizable and right dualizable are equivalent (see for example \cite[Proposition 1.2]{EfimovLimit}).
We also remind the reader that the $\Et$-monoidal structure on $\EE$ induces an $\Et$-monoidal structure on $\EEc$ by \cref{Ec-E2-monoidal}.

\begin{prop}\label{mod-like-rigid}
	Let $\BB \in \Alg(\EE)$, then a left $\BB$-module $M$ is compact if and only if it is left dualizable.
\end{prop}

\begin{proof}
	Assume that $M$ is left dualizable.
	Then,
	\[
		\hom^{\Sp}_{\BB}(M, -) \simeq \hom^{\Sp}_{\EE}(\AA, \hom^{\EE}_{\BB}(M, -)) \simeq \hom^{\Sp}_{\EE}(\AA, \ldual{M} \otimes_{\BB} -),
	\]
	and $\AA$ is compact in $\EE$ by rigidity, hence the right hand side preserves filtered colimits, i.e.\ $M$ is compact.

	For the other direction, note that since $C_\alpha$ is dualizable in $\EE$, the free $\BB$-module $\BB \otimesA C_\alpha$ is left dualizable in $\Mod_{\BB}(\EE)$.
	The conclusion then follows from the fact that these are compact generators, and left dualizable objects form a thick subcategory.
\end{proof}

As a consequence, we deduce the following two propositions.

\begin{prop}\label{lim-dbl}
	The equivalence from \cref{ordinary-limit} restricts to an equivalence
	\[
		\EEc^\dbl \iso \lim_\ii \Mod_{\AA/\xx^\ii}^\omega.
	\]
\end{prop}

\begin{proof}
	Recall from \cref{auto-complete} that any $\AA/\xx^\ii$-module is $\xx$-complete, so the equivalence from \cref{ordinary-limit} can be written as
	\[
		\EEc \iso \lim_\ii \cMod_{\AA/\xx^\ii}.
	\]
	Since $\EE$ is assumed to be $\Et$-monoidal, the same is true for right modules.
	Thus, we see that it restricts to an equivalence on left dualizable objects, giving
	\[
		\EEc^\dbl \iso \lim_\ii \cMod_{\AA/\xx^\ii}^\dbl.
	\]
	Using \cref{auto-complete} in reverse, we see that dualizable $\xx$-complete $\AA/\xx^\ii$-modules are the same as dualizable $\AA/\xx^\ii$-modules, which by \cref{mod-like-rigid} are the same as compact $\AA/\xx^\ii$-modules, giving the required equivalence.
\end{proof}

This also allows to compute the compactly generated limit.

\begin{prop}\label{lim-w}
	The limit computed in $\PrLstw$ is
	\[
		\Ind(\EEc^\dbl) \iso \limw[\ii] \Mod_{\AA/\xx^\ii}.
	\]
\end{prop}

\begin{proof}
	Since the forgetful from $\Catperf$ to $\Cat$ preserves all limits, the limit from \cref{lim-dbl} can be computed in $\Catperf$.
	Thus, the result follows from the equivalence (c.f.\ \cite[Lemma 5.3.2.9]{HA})
	\[
		\Ind\colon \Catperf \adj \PrLstw \noloc (-)^\omega.
	\]
\end{proof}

This also has the following consequence.

\begin{prop}\label{compact-dbl}
	Any compact $\xx$-complete object of $\EE$ is dualizable, that is
	\[
		\EEc^\omega \quad\subset\quad \EEc^\dbl.
	\]
\end{prop}

\begin{proof}
	Recall the identification of the ordinary limit
	\[
		\EEc \iso \lim_\ii \Mod_{\AA/\xx^\ii},
		\qquad M \mapstoo (M \otimesA \AA/\xx^\ii)_\ii
	\]
	from \cref{ordinary-limit}.
	Now, since the tower is strong Mittag-Leffler (\cref{ML-quotient-Sp}), the projections are strongly continuous, hence preserve compact objects, by \cite[Proposition 5.5(i)]{EfimovLimit}.
	Thus, if $M \in \EEc$ is compact, then each of the modules $M \otimesA \AA/\xx^\ii$ is compact in $\Mod_{\AA/\xx^\ii}$, hence $M$ is dualizable by \cref{lim-dbl}.
\end{proof}

We now summarize the comparison between the different limits computed in
\[
	\PrLstw \quad\subset\quad \PrLstdbl \quad\subset\quad \PrLst,
\]
and the different maps between them.

\begin{defn}
	We denote the assembly map for the first inclusion by
	\[
		\mdef{\varphi}\colon \Ind(\EEc^\dbl) \too \NucE
		\qin \PrLstdbl.
	\]
	We also denote the Ind-completion of the fully faithful functor from \cref{compact-dbl} by
	\[
		\mdef{j}\colon \EEc \hoook \Ind(\EEc^\dbl)
		\qin \PrLstw.
	\]
\end{defn}

\begin{prop}\label{three-lims}
	The ordinary, dualizable and compactly generated limits of $(\Mod_{\AA/\xx^\bullet})$, together with the assembly maps between them, fit into the solid commutative square below.
	The square is vertically left adjointable, and the functors $j$, $\varphi$ and $\iota$ are fully faithful.
	% https://q.uiver.app/#q=WzAsNCxbMCwwLCJcXEluZChcXEVFY15cXGRibCkiXSxbMSwwLCJcXE51Y0UiXSxbMSwxLCJcXEVFYyJdLFswLDEsIlxcRUVjIl0sWzMsMiwiIiwwLHsibGV2ZWwiOjIsInN0eWxlIjp7ImhlYWQiOnsibmFtZSI6Im5vbmUifX19XSxbMCwxLCJcXHZhcnBoaSIsMCx7InN0eWxlIjp7InRhaWwiOnsibmFtZSI6Imhvb2siLCJzaWRlIjoidG9wIn19fV0sWzEsMiwicSJdLFswLDNdLFsyLDEsIlxcaW90YSIsMCx7Im9mZnNldCI6LTEsImN1cnZlIjotMSwic3R5bGUiOnsidGFpbCI6eyJuYW1lIjoiaG9vayIsInNpZGUiOiJ0b3AifSwiYm9keSI6eyJuYW1lIjoiZGFzaGVkIn19fV0sWzMsMCwiaiIsMCx7Im9mZnNldCI6LTEsImN1cnZlIjotMSwic3R5bGUiOnsidGFpbCI6eyJuYW1lIjoiaG9vayIsInNpZGUiOiJ0b3AifSwiYm9keSI6eyJuYW1lIjoiZGFzaGVkIn19fV0sWzksNywiIiwwLHsibGV2ZWwiOjEsInN0eWxlIjp7Im5hbWUiOiJhZGp1bmN0aW9uIn19XSxbOCw2LCIiLDAseyJsZXZlbCI6MSwic3R5bGUiOnsibmFtZSI6ImFkanVuY3Rpb24ifX1dXQ==
	\[\begin{tikzcd}[sep=large]
		{\Ind(\EEc^\dbl)} & \NucE \\
		\EEc & \EEc
		\arrow["\varphi", hook, from=1-1, to=1-2]
		\arrow[""{name=0, anchor=center, inner sep=0}, from=1-1, to=2-1]
		\arrow[""{name=1, anchor=center, inner sep=0}, "q", from=1-2, to=2-2]
		\arrow[""{name=2, anchor=center, inner sep=0}, "j", shift left, curve={height=-6pt}, dashed, hook, from=2-1, to=1-1]
		\arrow[equals, from=2-1, to=2-2]
		\arrow[""{name=3, anchor=center, inner sep=0}, "\iota", shift left, curve={height=-6pt}, dashed, hook, from=2-2, to=1-2]
		\arrow["\dashv"{anchor=center}, draw=none, from=2, to=0]
		\arrow["\dashv"{anchor=center}, draw=none, from=3, to=1]
	\end{tikzcd}\]
\end{prop}

\begin{proof}
	The identification of the ordinary limit is \cref{ordinary-limit}, the dualizable limit is by \cref{nuc-def}, and the compactly generated limit is \cref{lim-w}.
	The commutativity of the solid square is the fact that assembly maps compose.
	
	To see that $\varphi$ is fully faithful, recall that $\PrLstw \hookrightarrow \PrLstdbl$ has a right adjoint, given by $\Ind((-)^\omega)$, i.e.\ the full subcategory generated by the compact objects (see for example \cite[Observation 1.28 and Corollary 3.15]{DualCat} and \cite[Proof of Theorem 1.91]{EfimovK}).
	Since it is a right adjoint, it takes the dualizable limit to the compactly generated limit, and exhibits it as the inclusion of the full subcategory generated by the compact objects.

	The fact that the left adjoint $j$ is fully faithful is \cref{compact-dbl}, and the fact that $\iota$ is fully faithful is \cref{iota}.

	The adjointability of the square follows from the fact that both $\iota$ and $j$ were constructed by the observation that the universal cone
	\[
		\EEc \too (\Mod_{\AA/\xx^\bullet})
	\]
	exhibiting the source as the ordinary limit of the target, is a cone in $\PrLstw$, hence also in $\PrLstdbl$, so it factors through the compactly generated and dualizable limits.
\end{proof}

Now, we turn our attention to the Verdier quotient of the inclusion of compacts in dualizables.
\cref{compact-dbl} provides one such inclusion, but in fact, there is a second inclusion.
Consider the inclusion (see \cref{non-complete-modules})
\[
	\EEc = \cMod_{\xc{\AA}} \quad\subset\quad \Mod_{\xc{\AA}}.
\]
We show that any compact $\xc{\AA}$-module is itself $\xx$-complete, and dualizable among them.

\begin{prop}\label{compact-mod-complete-dualizable}
	There is an inclusion
	\[
		\Mod_{\xc{\AA}}^\omega \quad\subset\quad \EEc^\dbl.
	\]
\end{prop}

\begin{proof}
	Recall that $\{\xc{\AA} \otimesA C_\alpha\}$ are compact generators of $\xc{\AA}$-modules, so that $\Mod_{\xc{\AA}}^\omega$ is the thick subcategory generated by them.
	Since dualizable objects form a thick subcategory, it suffices to show that each $\xc{\AA} \otimesA C_\alpha$ is $\xx$-complete and left dualizable in $\EEc$.

	Since $\EE$ is rigid, $C_\alpha$ is left dualizable in $\EE$, so tensoring with it commutes with limits and we get
	\[
		\xc{\AA} \otimesA C_\alpha
		\simeq (\lim_\ii \AA/\xx^\ii) \otimesA C_\alpha
		\simeq \lim_\ii (\AA/\xx^\ii \otimesA C_\alpha)
		\simeq \xc{C_\alpha}
	\]
	That is, $\xc{\AA} \otimesA C_\alpha$ is the $\xx$-completion of $C_\alpha$, which in particular shows it is $\xx$-complete, and since $\xx$-completion is monoidal (\cref{Ec-E2-monoidal}), we also learn that it is dualizable.
\end{proof}

The following argument is similar to \cite[Proposition 4.15]{DescVan}.

\begin{prop}\label{verdier-complete-compact-dualizable}
	The Verdier quotient of the inclusions
	\[
		\EEc^\omega \hoook \EEc^\dbl,
		\qquad \Mod_{\xc{\AA}}^\omega \hoook \EEc^\dbl
	\]
	are $\Ax$-acyclic.
	That is, for every $M \in \EEc^\dbl$, the tensor product $\Ax \otimesA M$ is compact both in $\EEc$ and in $\xc{\AA}$-modules.
\end{prop}

\begin{proof}
	Recall the equivalence from \cref{lim-dbl}
	\[
		\EEc^\dbl \iso \lim_\ii \Mod_{\AA/\xx^\ii}^\omega,
		\qquad M \mapstoo (\AA/\xx^\ii \otimesA M)_\ii.
	\]
	Considering the bottom factor $\Ax = \AA/\xx$, this shows that $\Ax \otimesA M \in \Mod_{\Ax}^\omega$ is a compact $\Ax$-module.
	Since $\Mod_{\Ax} = \cMod_{\Ax}$ by \cref{auto-complete}, it is also compact as an $\xx$-complete $\Ax$-module.
	The result follows from \cref{quotient-compact} by the two restriction of scalars along $\xc{\AA} \to \Ax$.
\end{proof}

	\section{Continuity along chromatic completion}

\subsection{Redshift preliminaries}\label{subsec-redshift}

In the next subsection we will apply the results of the previous sections to the chromatic setting.
In doing so, we will use the purity theorem of Land--Mathew--Meier--Tamme \cite{purity}, and the redshift upper bound of Clausen--Mathew--Naumann--Noel \cite{DescVan} (see also \cite{redshiftQL} for another proof by the author).
In particular, we need to know that these hold for dualizable categories and continuous K-theory, and we record these mild extensions here.
We start with the following variant of the purity theorem.

\begin{prop}[{\cite[Purity Theorem]{purity}}]\label{purity-variant}
	Let $F\colon \CC \to \CC' \in \Catperf$ be an exact functor between perfect categories, whose essential image thickly generates $\CC'$.
	Fix some number $m \geq 1$, and assume that for every $M \in \CC$, the induced map on endomorphism spectra
	\[
		\End_\CC(M) \too \End_{\CC'}(F(M))
	\]
	is a $(\Tmn\oplus\Tm)$-local isomorphism.
	Then $F$ induces an isomorphism on $\Tm$-localized K-theory
	\[
		\LTm\KK(\CC) \iso \LTm\KK(\CC').
	\]
	More generally, for any dualizable $\DD \in \PrLstdbl$, the following is an isomorphism
	\[
		\LTm\Kcont(\Ind(\CC) \otimes \DD) \iso \LTm\Kcont(\Ind(\CC') \otimes \DD).
	\]
\end{prop}

\begin{remark}\label{end-hom}
	We note that the condition of being a $(\Tmn\oplus\Tm)$-local isomorphism on endomorphisms is equivalent to being a $(\Tmn\oplus\Tm)$-local isomorphism on the hom spectrum for any pair of objects $M,N \in \CC$.
	The latter clearly implies the former.
	For the converse, note that
	\[
		\End(M \oplus N) \simeq \End(M) \oplus \hom(M,N) \oplus \hom(N,M) \oplus \End(N),
	\]
	so $\hom(M,N)$ is a retract of $\End(M \oplus N)$, and a retract of an isomorphism is an isomorphism.

	Since $\hom$ is exact in each coordinate separately, the latter condition can be checked where $M$ and $N$ range over thick generators of $\CC$.
\end{remark}

\begin{proof}[{Proof of \cref{purity-variant}}]
	We begin with the first part, without an extra category $\DD$.
	By the Schwede--Shipley theorem \cite{SchwedeShipley} (see for example \cite[Proposition 2.9]{Desc} for the following formulation), $\CC$ can be written as the filtered colimit
	\[
		\colim_{\CC_0 \in P} \CC_0 \iso \CC,
	\]
	where $P$ is the poset with respect to inclusion of subcategories $\CC_0 \subset \CC$ thickly generated by a single object, i.e.\ of the form $\CC_0 = \Thick(M)$.
	Moreover, by the assumption that the essential image of $F$ thickly generates $\CC'$, we conclude that we also have an isomorphism
	\[
		\colim_{\CC_0 \in P} \Thick(F(\CC_0)) \iso \CC'.
	\]
	Since K-theory commutes with filtered colimits, and $\Tm$-localization commutes with any colimits, we are therefore reduced to the case of a single $\CC_0$, i.e.\ showing that
	\[
		\LTm \KK(\CC_0) \too \LTm \KK(\Thick(F(\CC_0)))
	\]
	is an isomorphism.
	Choosing a generator $M$ for $\CC_0$, this map identifies with
	\[
		\LTm \KK(\End_\CC(M)) \too \LTm \KK(\End_{\CC'}(F(M))).
	\]
	By assumption, the map on endomorphism spectra is a $(\Tmn\oplus\Tm)$-local isomorphism, so this map is an isomorphism by \cite[Purity Theorem]{purity}.

	We now handle the case where $\DD$ is compactly generated.
	We claim that
	\[
		F \otimes \Id_{\DD^\omega}\colon \CC \otimes \DD^\omega \too \CC' \otimes \DD^\omega
	\]
	satisfies the conditions of the proposition.
	Note that $\CC \otimes \DD^\omega$ is thickly generated by objects of the form $M \otimes X$ for $M \in \CC$ and $X \in \DD^\omega$, and similarly for $\CC' \otimes \DD^\omega$.
	This shows that the essential image of $F \otimes \Id_{\DD^\omega}$ thickly generates $\CC' \otimes \DD^\omega$, verifying the first condition.
	For the second condition, recall that by \cref{end-hom}, we can check the condition on hom spaces for thick generators, that is, we need to show that for $M, N \in \CC$ and $X, Y \in \DD^\omega$, the map
	\[
		\hom_{\CC \otimes \DD^\omega}(M \otimes X, N \otimes Y) \too \hom_{\CC' \otimes \DD^\omega}(F(M) \otimes X, F(N) \otimes Y)
	\]
	is a $(\Tmn\oplus\Tm)$-local isomorphism.
	Note that this map is identified with
	\[
		\hom_\CC(M, N) \otimes \hom_{\DD^\omega}(X, Y) \too \hom_{\CC'}(F(M), F(N)) \otimes \hom_{\DD^\omega}(X, Y).
	\]
	Using \cref{end-hom} again, we see that $\hom_\CC(M, N) \too \hom_{\CC'}(F(M), F(N))$ is a $(\Tmn\oplus\Tm)$-local isomorphism, and the conclusion follows by compatibility with the tensor product.

	Finally, we prove the case of general dualizable categories.
	Consider the natural transformation
	\[
		\LTm\Kcont(\Ind(\CC) \otimes -) \too \LTm\Kcont(\Ind(\CC') \otimes -)
	\]
	of functors from $\PrLstdbl$ to $\SpTm$.
	Observe that tensoring with $\Ind(\CC)$ or $\Ind(\CC')$ preserves Verdier sequences by \cite[Theorem 2.2]{EfimovK}, so, since $\Kcont$ is a localizing invariant, both of these functors are localizing invariants.
	By \cite[Theorem 0.1]{EfimovK}, the category of localizing invariants of dualizable categories is equivalent to the category of localizing invariants of perfect categories, so we are done by the previous case.
\end{proof}

We now wish to extend the redshift upper bound.
To do this, we need to know that continuous K-theory is lax symmetric monoidal.
This is certainly well-known, but unfortunately we could not find a reference for the precise statement, and we include a proof for completeness.

\begin{prop}
	Continuous K-theory is given by
	\[
		\Kcont \simeq \hom_{\Motloc}(\Uloccont(\Sp), \Uloccont(-)).
	\]
	where
	\[
		\Uloccont\colon \PrLstdbl \too \Motloc
	\]
	is the continuous extension of the universal localizing invariant.
\end{prop}

\begin{proof}
	By \cite[Theorem 1.3]{BGT1}, we have
	\[
		\KK(-) \simeq \hom_{\Motloc}(\Uloc(\Perf(\SS)), \Uloc(-)).
	\]
	Recall that by construction $\Uloccont(\Sp) \simeq \Uloc(\Perf(\SS))$, and so
	\[
		\hom_{\Motloc}(\Uloccont(\Sp), \Uloccont(-))
	\]
	is a continuous extension of the right hand side.
	Recall that \cite[Theorem 0.1]{EfimovK} says that every localizing invariant has a \emph{unique} continuous extension, hence it is also the continuous extension of $\KK(-)$.
\end{proof}

\begin{prop}\label{Kcont-lax-sm}
	The functor
	\[
		\Kcont\colon \PrLstdbl \too \Sp
	\]
	is lax symmetric monoidal.
\end{prop}

\begin{proof}
	As explained in \cite[Remark 4.15]{EfimovK} (building on \cite[Theorem 1.4]{BGT2}) the functor $\Uloccont$ is symmetric monoidal.
	In particular, $\Kcont$ is given by composing the symmetric monoidal functor $\Uloccont$ with the lax symmetric monoidal functor of hom from the unit, and the result follows.
\end{proof}

This then immediately implies the following extension.

\begin{prop}[{\cite[Theorem C]{DescVan}}]\label{vanishing}
	Let $\DD \in \PrLstdbl$ be $\Lnmf$-local (that is, a module over $\Lnmf\Sp$).
	Then for any $m \geq n{+}1$ we have
	\[
		\LTm\Kcont(\DD) = 0.
	\]
	Consequently, if $\CC \hookrightarrow \CC' \in \PrLstdbl$ is a fully faithful strongly continuous functor between dualizable categories whose Verdier quotient $\CC'/\CC$ is $\Lnmf$-local, then the induced map
	\[
		\Kcont(\CC) \too \Kcont(\CC')
	\]
	is a $\Tm$-local isomorphism for $m \geq n{+}1$.
\end{prop}

\begin{proof}
	For the first part, $\Kcont(\DD)$ is a module over $\Kcont(\Lnmf\Sp) \simeq \KK(\Lnmf\SS)$ since $\Kcont$ is lax symmetric monoidal (\cref{Kcont-lax-sm}), and by \cite[Theorem C]{DescVan} the latter is $\Tm$-acyclic.
	The second part follows from the fact that $\Kcont$ is a localizing invariant.
\end{proof}

\subsection{Chromatic comparison}\label{subsec-chrom}

In this subsection we apply the results of the previous section to the chromatic setting.
We will in particular work in \assref{setup} and \assref{comp-assum} throughout, and assume rigidity explicitly when needed.
We will assume moreover that (the underlying $\Et$-ring spectrum of) $\AA$ is complex orientable, which for example holds if it is even.
Recall that $(p,v_1,\dotsc,v_{n-1}) \subseteq \pi_*(\MU)$ is an invariant ideal, hence its image in $\pi_*(\AA)$ is independent of the choice of complex orientation, so we get a well-defined height $n$ ideal in $\pi_*(\AA)$.
Thus, the setup for this subsection is the following.

\begin{namedassum}{Chromatic Setup}\label{chromatic-assum}
    We fix a presentably $\Et$-monoidal stable category $\EE \in \Alg(\PrLst)$ and an $\Et$-map
    \[
        \SS[\xx] := \SS[x_1] \otimes \cdots \otimes \SS[x_r] \too \End(\AA)
        \qin \Alg_{\Et}(\Sp).
    \]
	We assume that the unit is complex orientable, and the elements lie in the radical of the height $n$ ideal
	\[
		x_1,\dotsc,x_r \ \in\ \sqrt{(p,v_1,\dotsc,v_{n-1})}
		\qquad\subseteq\qquad \pi_*(\AA).
	\]
\end{namedassum}

\begin{example}
	As in \cref{ex-of-setup}, this (and rigidity) holds for $\EE = \Mod_R$, where $R$ is an even $\Eth$-ring spectrum, and elements of non-negative even degree $x_1,\dotsc,x_r$ in the radical of the height $n$ ideal.
\end{example}

We start with the following observation, which gives the connection to chromatic localizations.

\begin{prop}\label{moore-Ax}
	For any type $n$ generalized Moore spectrum
	\[
		F = \SS/(p^{e_0},v_1^{e_1},\dotsc,v_{n-1}^{e_{n-1}}),
	\]
	the tensor product $\AA \otimes F$ is in the thick subcategory of $\EE$ generated by $\Ax$.
\end{prop}

\begin{proof}
	We will show that $\AA \otimes F$ is a retract of $\AA/\xx^\ii \otimes F$ for $\ii$ large enough.
	This will imply the claim, since $\AA/\xx^\ii$ is in the thick subcategory generated by $\Ax$ by \cref{thick}, and $F$ is a finite spectrum, so $\AA/\xx^\ii \otimes F$ is in the thick subcategory generated by $\Ax$, hence the same will follow for $\AA \otimes F$.
	
	We begin by showing that all elements of the ideal $(p^{e_0},v_1^{e_1},\dotsc,v_{n-1}^{e_{n-1}}) \subseteq \pi_*(\AA)$ are null in $\AA \otimes F$.
	It suffices to show this for the generators of the ideal.
	Since $\AA$ is complex oriented, it is in particular an $\MU$-module in the homotopy category.
	Thus, by \cite[Theorem 4.6]{HSt}, the following triangle commutes
	% https://q.uiver.app/#q=WzAsMyxbMCwwLCJcXFNpZ21hXntlX218dl9tfH0gUiBcXG90aW1lcyBcXFNTLyh7cF57ZV8wfSxcXGRvdHNjLHZfe20tMX1ee2Vfe20tMX19fSkiXSxbMCwyLCJSIFxcb3RpbWVzIFxcU2lnbWFee2VfbXx2X218fSBcXFNTLyh7cF57ZV8wfSxcXGRvdHNjLHZfe20tMX1ee2Vfe20tMX19fSkiXSxbMSwxLCJSIFxcb3RpbWVzIFxcU1MvKHtwXntlXzB9LFxcZG90c2Msdl97bS0xfV57ZV97bS0xfX19KSJdLFswLDEsIlxcd3IiLDJdLFswLDIsInZfbV57ZV9tfSBcXG90aW1lcyBcXElkIl0sWzEsMiwiXFxJZCBcXG90aW1lcyB2X21ee2VfbX0iLDJdXQ==
	\[\begin{tikzcd}
		{\Sigma^{e_m|v_m|} \AA \otimes \SS/({p^{e_0},\dotsc,v_{m-1}^{e_{m-1}}})} & \\
		& {\AA \otimes \SS/({p^{e_0},\dotsc,v_{m-1}^{e_{m-1}}})} \\
		{\AA \otimes \Sigma^{e_m|v_m|} \SS/({p^{e_0},\dotsc,v_{m-1}^{e_{m-1}}})}
		\arrow["{v_m^{e_m} \otimes \Id}", from=1-1, to=2-2]
		\arrow["\wr"', from=1-1, to=3-1]
		\arrow["{\Id \otimes v_m^{e_m}}"', from=3-1, to=2-2]
	\end{tikzcd}\]
	We see that $v_m^{e_m}$ indeed is null already in $\AA \otimes \SS/({p^{e_0},\dotsc,v_m^{e_m}})$.

	Now, take $e$ large enough, such as $1{+}e_0{+}\cdots{+}e_{n-1}$, so that
	\[
		(p,v_1,\dotsc,v_{n-1})^e \subseteq (p^{e_0},v_1^{e_1},\dotsc,v_{n-1}^{e_{n-1}}) \subseteq \pi_*(\AA).
	\]
	By assumption the elements $x_1,\dotsc,x_r$ are in the radical of the height $n$ ideal, that is, for some $k$ we have $x_1^k,\dotsc,x_r^k \in (p,v_1,\dotsc,v_{n-1})$.
	Combining these, we get that for $i \geq ek$ we have
	\[
		x_1^i,\dotsc,x_r^i \in (p^{e_0},v_1^{e_1},\dotsc,v_{n-1}^{e_{n-1}}) \subseteq \pi_*(\AA),
	\]
	so these are null in $\AA \otimes F$.
	Therefore, for $\ii \geq (ek,\dotsc,ek)$, we get that $\AA/\xx^\ii \otimes F$ is a direct sum of suspensions of $\AA \otimes F$, so, in particular, $\AA \otimes F$ is indeed a retract of $\AA/\xx^\ii \otimes F$.
\end{proof}

Recall that being $\Lnmf$-local is, by definition, the same as being acyclic with respect to some (equivalently, any) type $n$ finite spectrum, so we immediately get the following.

\begin{cor}\label{Ax-Lnmf}
	If $\CC \in \PrLE$ is $\Ax$-acyclic then it is $\Lnmf$-local.
\end{cor}

\begin{proof}
	By assumption, $\CC$ is $\Ax$-acyclic, hence by \cref{moore-Ax}, it is $(\AA \otimes F)$-acyclic for some type $n$ generalized Moore spectrum $F$ (relative to $\EE$).
	Namely, it is $F$-acyclic (relative to $\Sp$), that is, $\Lnmf$-local.
\end{proof}

With this in place, we deduce the following theorem.

\begin{thm}\label{main-thm}
	The following maps are $\Tm$-local isomorphisms for $m \geq n{+}1$
	\[
		\Kcont(\EEc) \too \Kcont(\NucE) \iso \lim_\ii \Kcont(\Mod_{\AA/\xx^\ii}).
	\]
	More generally, for any dualizable $\DD \in \PrLstdbl$, the same is true for
	\[
		\Kcont(\EEc \otimes \DD) \too \Kcont(\NucE \otimes \DD) \iso \lim_\ii \Kcont(\Mod_{\AA/\xx^\ii} \otimes \DD).
	\]
\end{thm}

\begin{proof}
	The second map is an isomorphism by \cref{k-nuc}.
	For the first map, recall from \cref{comp-nuc} that the Verdier quotient of $\iota\colon \EEc \hookrightarrow \NucE$ is $\Ax$-acyclic, so by \cref{Ax-Lnmf} it is $\Lnmf$-local.
	By \cite[Theorem 2.2]{EfimovK}, the tensor product
	\[
		\iota \otimes \Id_{\DD}\colon \EEc \otimes \DD \too \NucE \otimes \DD
	\]
	is fully faithful and strongly continuous as well, and its Verdier quotient is $(\Nuc/\EEc) \otimes \DD$, which is therefore $\Lnmf$-local.
	The result then follows from \cref{vanishing}.
\end{proof}

As examples, we apply this to the categories of functors and sheaves with coefficients in $\EE$.

\begin{example}[Functor categories]\label{fun-comp}
	Continuing \cref{fun-nuc}, by \cref{main-thm}, for a (small) category $\CC$, the following is a $\Tm$-local isomorphism for $m \geq n{+}1$
	\[
		\Kcont(\Fun(\CC, \EEc)) \too \lim_\ii \Kcont(\Fun(\CC, \Mod_{\AA/\xx^\ii})).
	\]
\end{example}

\begin{example}[Sheaves]\label{sh-comp}
	Continuing \cref{sh-nuc}, by \cref{main-thm}, for a locally compact Hausdorff space $Y$, the following is a $\Tm$-local isomorphism for $m \geq n{+}1$
	\[
		\Kcont(\Shv(Y, \EEc)) \too \lim_\ii \Kcont(\Shv(Y, \Mod_{\AA/\xx^\ii})).
	\]
	which \cite[Theorem 0.2]{EfimovK} identifies with
	\[
		\Gammac(Y, \Kcont(\EEc)) \too \lim_\ii \Gammac(Y, \Kcont(\Mod_{\AA/\xx^\ii})).
	\]
\end{example}

Our next goal is to show that in the rigid compactly generated case, we can also connect this to the K-theory of the compact objects of $\EE$.
Consider the functor of extension of scalars along $\AA \to \xc{\AA}$
\[
	\EE \too \Mod_{\xc{\AA}}.
\]
Since it is strongly continuous, it in particular sends compact objects to compact objects, and we prove the following result.

\begin{prop}\label{can-assume-complete}
	If $\EE$ is compactly generated the following map is a $\Tm$-local isomorphism for $m \geq n{+}1$
	\[
		\KK(\EE^\omega) \too \KK(\Mod_{\xc{\AA}}^\omega).
	\]
	More generally, for any dualizable $\DD \in \PrLstdbl$, the same holds for
	\[
		\Kcont(\EE \otimes \DD) \too \Kcont(\Mod_{\xc{\AA}} \otimes \DD).
	\]
\end{prop}

\begin{proof}
	We verify the conditions of \cref{purity-variant}.
	First, we observe that the essential image of
	\[
		\EE^\omega \too \Mod_{\xc{\AA}}^\omega
	\]
	thickly generates the target.
	Indeed, by assumption $\EE$ is compactly generated by $\{C_\alpha\}$, which implies that $\{\xc{\AA} \otimesA C_\alpha\}$ are compact generators of $\Mod_{\xc{\AA}}$.

	It remains to show the second condition.
	Note that the collection of objects $X \in \EE$ such that
	\[
		X \otimes \AA \too X \otimes \xc{\AA}
		\qin \EE
	\]
	is an isomorphism, forms a thick subcategory, which, since $\AA \to \xc{\AA}$ is an $\xx$-completion, contains $\Ax$.
	By \cref{moore-Ax}, it also contains $\AA \otimes F$, for some type $n$ generalized Moore spectrum, so the following map is an isomorphism
	\[
		F \otimes \AA \iso F \otimes \xc{\AA}.
	\]
	Now, let $M \in \EE^\omega$.
	Tensoring the isomorphism above, we get an isomorphism
	\[
		F \otimes M \iso F \otimes \xc{\AA} \otimes M.
	\]
	Taking hom from $M$, and using that $F$ is finite, we get
	\[
		F \otimes \hom_{\EE}(M, M)
		\simeq \hom_{\EE}(M, F \otimes M)
		\iso \hom_{\EE}(M, F \otimes \xc{\AA} \otimes M)
		\simeq F \otimes \hom_{\EE}(M, \xc{\AA} \otimes M).
	\]
	By free-forgetful adjunction, this can be rewritten as
	\[
		F \otimes \End_{\EE}(M)
		\iso F \otimes \End_{\Mod_{\xc{\AA}}}(\xc{\AA} \otimesA M),
	\]
	namely the map on endomorphism spectra is an $F$-local isomorphism.
	Thus, for $m \geq n{+}1$, this map is in particular a $(\Tmn\oplus\Tm)$-local isomorphism, so the result follows from \cref{purity-variant}.
\end{proof}

\begin{thm}[{\cref{main-thm-rcg-intro}}]\label{main-thm-rcg}
	If $\EE$ is rigid and compactly generated then the following maps are $\Tm$-local isomorphisms for $m \geq n{+}1$
	% https://q.uiver.app/#q=WzAsNixbMSwwLCJcXEtLKFxcTW9kX3tcXHhje1xcQUF9fV5cXG9tZWdhKSJdLFsyLDAsIlxcS0soXFxFRWNeXFxkYmwpIl0sWzMsMCwiXFxLY29udChcXE51Y0UpIl0sWzIsMSwiXFxLSyhcXEVFY15cXG9tZWdhKSJdLFswLDAsIlxcS0soXFxFRV5cXG9tZWdhKSJdLFs0LDAsIlxcbGltX1xcaWkgXFxLSyhcXE1vZF97XFxBQS9cXHh4XlxcaWl9Xlxcb21lZ2EpIl0sWzAsMV0sWzEsMl0sWzMsMV0sWzQsMF0sWzMsMiwiIiwyLHsiY3VydmUiOjJ9XSxbMiw1LCJcXHNpbSJdXQ==&macro_url=https%3A%2F%2Fgist.githubusercontent.com%2Fshaybenmoshe%2F301102e6fdd6d215848c519c149abcab%2Fraw%2Fquiver
	\[\begin{tikzcd}[row sep=scriptsize]
		{\KK(\EE^\omega)} & {\KK(\Mod_{\xc{\AA}}^\omega)} & {\KK(\EEc^\dbl)} & {\Kcont(\NucE)} & {\lim_\ii \KK(\Mod_{\AA/\xx^\ii}^\omega)} \\
		&& {\KK(\EEc^\omega)}
		\arrow[from=1-1, to=1-2]
		\arrow[from=1-2, to=1-3]
		\arrow[from=1-3, to=1-4]
		\arrow["\sim", from=1-4, to=1-5]
		\arrow[from=2-3, to=1-3]
		\arrow[curve={height=12pt}, from=2-3, to=1-4]
	\end{tikzcd}\]
	More generally, for any dualizable $\DD \in \PrLstdbl$, the same is true for
	% https://q.uiver.app/#q=WzAsNixbMCwwLCJcXEtjb250KFxcRUUge1xcb3RpbWVzfSBcXEREKSJdLFsxLDAsIlxcS2NvbnQoXFxNb2Rfe1xceGN7XFxBQX19IHtcXG90aW1lc30gXFxERCkiXSxbMiwwLCJcXEtjb250KFxcSW5kKFxcRUVjXlxcZGJsKSB7XFxvdGltZXN9IFxcREQpIl0sWzMsMCwiXFxLY29udChcXE51Y0Uge1xcb3RpbWVzfSBcXEREKSJdLFs0LDAsIlxcbGltX1xcaWkgXFxLY29udChcXE1vZF97XFxBQS9cXHh4XlxcaWl9IHtcXG90aW1lc30gXFxERCkiXSxbMiwxLCJcXEtjb250KFxcRUVjIHtcXG90aW1lc30gXFxERCkiXSxbMCwxXSxbMSwyXSxbMiwzXSxbMyw0LCJcXHNpbSJdLFs1LDJdLFs1LDMsIiIsMCx7ImN1cnZlIjoyfV1d
	\[\begin{tikzcd}[cramped,column sep=1em,row sep=scriptsize]
		{\Kcont(\EE {\otimes} \DD)} & {\Kcont(\Mod_{\xc{\AA}} {\otimes} \DD)} & {\Kcont(\Ind(\EEc^\dbl) {\otimes} \DD)} & {\Kcont(\NucE {\otimes} \DD)} & {\lim_\ii \Kcont(\Mod_{\AA/\xx^\ii} {\otimes} \DD)} \\
		&& {\Kcont(\EEc {\otimes} \DD)}
		\arrow[from=1-1, to=1-2]
		\arrow[from=1-2, to=1-3]
		\arrow[from=1-3, to=1-4]
		\arrow["\sim", from=1-4, to=1-5]
		\arrow[from=2-3, to=1-3]
		\arrow[curve={height=12pt}, from=2-3, to=1-4]
	\end{tikzcd}\]
\end{thm}

\begin{proof}
	The left-most horizontal morphism is the content of \cref{can-assume-complete}.
	The right-most horizontal morphism and the curved diagonal morphisms are the content of \cref{main-thm}.
	For the other morphisms, consider the following commutative diagram (see \cref{three-lims} and \cref{compact-mod-complete-dualizable})
	% https://q.uiver.app/#q=WzAsNCxbMCwwLCJcXE1vZF97XFx4Y3tcXEFBfX0iXSxbMSwwLCJcXEluZChcXEVFY15cXGRibCkiXSxbMiwwLCJcXE51Y0UiXSxbMSwxLCJcXEVFYyJdLFswLDEsIiIsMCx7InN0eWxlIjp7InRhaWwiOnsibmFtZSI6Imhvb2siLCJzaWRlIjoidG9wIn19fV0sWzEsMiwiXFx2YXJwaGkiLDAseyJzdHlsZSI6eyJ0YWlsIjp7Im5hbWUiOiJob29rIiwic2lkZSI6InRvcCJ9fX1dLFszLDEsImoiLDAseyJzdHlsZSI6eyJ0YWlsIjp7Im5hbWUiOiJob29rIiwic2lkZSI6InRvcCJ9fX1dLFszLDIsIlxcaW90YSIsMix7ImN1cnZlIjoyLCJzdHlsZSI6eyJ0YWlsIjp7Im5hbWUiOiJob29rIiwic2lkZSI6InRvcCJ9fX1dXQ==&macro_url=https%3A%2F%2Fgist.githubusercontent.com%2Fshaybenmoshe%2F301102e6fdd6d215848c519c149abcab%2Fraw%2Fquiver
	\[\begin{tikzcd}
		{\Mod_{\xc{\AA}}} & {\Ind(\EEc^\dbl)} & \NucE \\
		& \EEc
		\arrow[hook, from=1-1, to=1-2]
		\arrow["\varphi", hook, from=1-2, to=1-3]
		\arrow["j", hook, from=2-2, to=1-2]
		\arrow["\iota"', curve={height=12pt}, hook, from=2-2, to=1-3]
	\end{tikzcd}\]
	By (the Ind of) \cref{verdier-complete-compact-dualizable}, the Verdier quotients of $j$ and the unlabeled morphism are $\Ax$-acyclic, hence by \cref{Ax-Lnmf} also $\Lnmf$-local.
	By \cite[Theorem 2.2]{EfimovK}, they remain fully faithful after tensoring with $\DD$, and since the Verdier quotient is given by tensoring with $\DD$, it remains $\Lnmf$-local as well.
	Therefore, by \cref{vanishing}, these maps induce an isomorphism on $\Tm$-localized K-theory.
	By $2$-out-of-$3$, the same is true for $\varphi$, concluding the proof.
\end{proof}

The main example satisfying \assref{chromatic-assum} is the case of ring spectra.

\begin{example}
	Let $R \in \Alg_{\Eth}(\Sp)$ be equipped with an $\Et$-map $\SS[\xx] \to R$, and assume that $x_1,\dotsc,x_r \in \pi_*(R)$ are in the radical of the height $n$ ideal $(p,v_1,\dotsc,v_{n-1}) \subset \pi_*(R)$.
	Then, \cref{main-thm-rcg} says that the following maps are $\Tm$-local isomorphisms for $m \geq n{+}1$
	% https://q.uiver.app/#q=WzAsNSxbMCwwLCJcXEtLKFIpIl0sWzEsMCwiXFxLSyhcXHhje1J9KSJdLFsyLDAsIlxcS0soXFxjTW9kX1JeXFxkYmwpIl0sWzMsMCwiXFxsaW1fXFxpaSBcXEtLKFIvXFx4eF5cXGlpKSJdLFsyLDEsIlxcS0soXFxjTW9kX1JeXFxvbWVnYSkiXSxbMCwxXSxbMSwyXSxbMiwzXSxbNCwyXV0=&macro_url=https%3A%2F%2Fgist.githubusercontent.com%2Fshaybenmoshe%2F301102e6fdd6d215848c519c149abcab%2Fraw%2Fquiver
	\[\begin{tikzcd}[row sep=scriptsize]
		{\KK(R)} & {\KK(\xc{R})} & {\KK(\cMod_R^\dbl)} & {\lim_\ii \KK(R/\xx^\ii)} \\
		&& {\KK(\cMod_R^\omega)}
		\arrow[from=1-1, to=1-2]
		\arrow[from=1-2, to=1-3]
		\arrow[from=1-3, to=1-4]
		\arrow[from=2-3, to=1-3]
	\end{tikzcd}\]
\end{example}

We repeat the examples of functor categories and sheaves from \cref{fun-comp} and \cref{sh-comp} in this case, taking $R$ as in the above example.

\begin{example}\label{sh-comp-R}
	For a locally compact Hausdorff space $Y$, the following maps are $\Tm$-local isomorphisms for $m \geq n{+}1$
	% https://q.uiver.app/#q=WzAsNSxbMCwwLCJcXEtjb250KFxcU2h2KFksIFIpKSJdLFsxLDAsIlxcS2NvbnQoXFxTaHYoWSwgXFx4Y3tSfSkpIl0sWzIsMCwiXFxLY29udChcXFNodihZLCBcXEluZChcXGNNb2RfUl5cXGRibCkpKSJdLFszLDAsIlxcbGltX1xcaWkgXFxLY29udChcXFNodihZLCBSL1xceHheXFxpaSkpIl0sWzIsMSwiXFxLY29udChcXFNodihZLCBcXGNNb2RfUl5cXG9tZWdhKSkiXSxbMCwxXSxbMSwyXSxbMiwzXSxbNCwyXV0=&macro_url=https%3A%2F%2Fgist.githubusercontent.com%2Fshaybenmoshe%2F301102e6fdd6d215848c519c149abcab%2Fraw%2Fquiver
	\[\begin{tikzcd}[row sep=scriptsize]
		{\Kcont(\Shv(Y, R))} & {\Kcont(\Shv(Y, \xc{R}))} & {\Kcont(\Shv(Y, \Ind(\cMod_R^\dbl)))} & {\lim_\ii \Kcont(\Shv(Y, R/\xx^\ii))} \\
		&& {\Kcont(\Shv(Y, \cMod_R^\omega))}
		\arrow[from=1-1, to=1-2]
		\arrow[from=1-2, to=1-3]
		\arrow[from=1-3, to=1-4]
		\arrow[from=2-3, to=1-3]
	\end{tikzcd}\]
\end{example}

\begin{example}\label{fun-comp-R}
	For a (small) category $\CC$, the following maps are $\Tm$-local isomorphisms for $m \geq n{+}1$
	% https://q.uiver.app/#q=WzAsNSxbMCwwLCJcXEtjb250KFxcRnVuKFxcQ0MsIFxcTW9kX1IpKSJdLFsxLDAsIlxcS2NvbnQoXFxGdW4oXFxDQywgXFxNb2Rfe1xceGN7Un19KSkiXSxbMiwwLCJcXEtjb250KFxcRnVuKFxcQ0MsIFxcSW5kKFxcY01vZF9SXlxcZGJsKSkpIl0sWzMsMCwiXFxsaW1fXFxpaSBcXEtjb250KFxcRnVuKFxcQ0MsIFIvXFx4eF5cXGlpKSkiXSxbMiwxLCJcXEtjb250KFxcRnVuKFxcQ0MsIFxcY01vZF9SXlxcb21lZ2EpKSJdLFswLDFdLFsxLDJdLFsyLDNdLFs0LDJdXQ==&macro_url=https%3A%2F%2Fgist.githubusercontent.com%2Fshaybenmoshe%2F301102e6fdd6d215848c519c149abcab%2Fraw%2Fquiver
	\[\begin{tikzcd}[cramped,column sep=1.1em,row sep=scriptsize]
		{\Kcont(\Fun(\CC, \Mod_R))} & {\Kcont(\Fun(\CC, \Mod_{\xc{R}}))} & {\Kcont(\Fun(\CC, \Ind(\cMod_R^\dbl)))} & {\lim_\ii \Kcont(\Fun(\CC, R/\xx^\ii))} \\
		&& {\Kcont(\Fun(\CC, \cMod_R^\omega))}
		\arrow[from=1-1, to=1-2]
		\arrow[from=1-2, to=1-3]
		\arrow[from=1-3, to=1-4]
		\arrow[from=2-3, to=1-3]
	\end{tikzcd}\]
\end{example}

Finally, in the connective setting we can also connect this to topological cyclic homology.

\begin{cor}[{\cref{ring-dgm-intro}}]\label{ring-dgm}
	Let $R \in \Alg_{\Eth}(\Spcn)$ be equipped with an $\Et$-map $\SS[\xx] \to R$, and assume that $x_1,\dotsc,x_r \in \pi_*(R)$ are in the radical of the height $n$ ideal $(p,v_1,\dotsc,v_{n-1}) \subset \pi_*(R)$.
	The following map is a $\Tm$-local isomorphism for $m \geq \max\{ n{+}1, 2\}$
	\[
		\KK(R) \too \lim_\ii \TC(R/\xx^\ii).
	\]
\end{cor}

\begin{proof}
	Combining \cref{main-thm-rcg} with \cref{nuc-dgm}, we see that the square
	% https://q.uiver.app/#q=WzAsNCxbMCwwLCJcXEtLKFIpIl0sWzAsMSwiXFxLSyhcXEF4KSJdLFsxLDAsIlxcbGltX1xcaWkgXFxUQyhSL1xceHheXFxpaSkiXSxbMSwxLCJcXFRDKFxcQXgpIl0sWzEsM10sWzIsM10sWzAsMV0sWzAsMl1d
	\[\begin{tikzcd}
		{\KK(R)} & {\lim_\ii \TC(R/\xx^\ii)} \\
		{\KK(\Ax)} & {\TC(\Ax)}
		\arrow[from=1-1, to=1-2]
		\arrow[from=1-1, to=2-1]
		\arrow[from=1-2, to=2-2]
		\arrow[from=2-1, to=2-2]
	\end{tikzcd}\]
	becomes a pullback square after $\Tm$-localization for any $m \geq n{+}1$.
	By \cite[Corollary 4.30]{purity}, the bottom map is a $\Tm$-local isomorphism for any $m \geq 2$, and the result follows.
\end{proof}

	% \newpage
	\bibliographystyle{alpha}
	\bibliography{refs}

\end{document}